\RequirePackage{fix-cm}

\documentclass[smallextended]{svjour3}       

\smartqed  
\usepackage{graphicx}
\usepackage{amsmath}
\DeclareMathOperator*{\argmin}{arg\,min}

\usepackage[utf8x]{inputenc}
\usepackage{cite}
\usepackage{url}
\usepackage{amsmath,amsfonts,amssymb}
\usepackage{fancyhdr}
\usepackage{algorithm}
\usepackage{algpseudocode}
\usepackage{accents}
\newcommand{\ubar}[1]{\underaccent{\bar}{#1}}

\usepackage{tkz-tab}
\usepackage{xpatch}
\usepackage{enumerate}   
\usepackage{multirow}
\usepackage{arydshln}
\usepackage{slashbox}
\usepackage{adjustbox}

\spnewtheorem{assumption}{Assumption}{\bf}{\it}
\spnewtheorem{corrolary}{Corrollary}{\bf}{\it}

\usepackage[title]{appendix}

\usepackage{verbatim}

\newcommand{\red}[1]{\textcolor{red}{#1}}

\DeclareUnicodeCharacter{2212}{-}

\begin{document}

\title{Complexity of linearized quadratic penalty  for  optimization with nonlinear equality constraints
\thanks{The research leading to these results has received funding from: the European Union's Horizon 2020 research and innovation programme under the Marie Sklodowska-Curie
grant agreement no. 953348;  UEFISCDI, Romania, PN-III-P4-PCE-2021-0720, under project L2O-MOC,
no. 70/2022. 
 }}

\titlerunning{Linearized quadratic penalty  for nonconvex optimization}        

\author{Lahcen El Bourkhissi         \and
        Ion Necoara }
        
\authorrunning{L. El Bourkhissi         \and
        I. Necoara}

\institute{Lahcen El Bourkhissi \at
             Department of Automatic Control and Systems Engineering, National University for Science and Technology Politehnica Bucharest, 
              Bucharest, 060042, Romania.\\
              \email{lel@stud.acs.upb.ro}           
           \and
           Ion Necoara \at
             Department of Automatic Control and Systems Engineering, National University for Science and Technology Politehnica Bucharest, 060042 Bucharest and 
             Gheorghe Mihoc-Caius Iacob Institute of Mathematical Statistics and Applied Mathematics of the Romanian Academy, 050711 Bucharest, Romania. \\
             \email{ion.necoara@upb.ro}
}

\date{Received: date / Accepted: date}

\maketitle

\begin{abstract}
In this paper we consider a nonconvex optimization
problem with nonlinear equality constraints. We assume
that both, the objective function and the functional constraints,
are locally smooth. For solving this problem, we propose a linearized quadratic penalty method, i.e., we linearize the objective function and the functional constraints in the penalty formulation at the current iterate and add a quadratic regularization, thus yielding a  subproblem that is easy to solve, and  whose solution is the next iterate. Under a new adaptive regularization parameter choice,  we provide convergence  guarantees for the iterates of this method to an $\epsilon$ first-order optimal solution  in $\mathcal{O}({\epsilon^{-2.5}})$  iterations. Finally, we  show that when the problem data satisfy Kurdyka-Lojasiewicz  property, e.g., are semialgebraic, the whole sequence generated by the proposed algorithm converges and we derive improved local convergence rates depending on the KL parameter. We  validate the theory and the performance of the proposed algorithm by numerically comparing it with some existing methods   from the literature.


\end{abstract}


\section{Introduction}
\label{intro}

In many fields, including machine learning, matrix optimization, statistics, control and signal processing, a variety of applications can be reformulated as nonconvex optimization problems with nonlinear functional equality constraints, see e.g.,  \cite{Fes:20, LukSab:19, HonHaj:17, Roy:19}. In this paper, we propose an approach to address this optimization problem using a penalty approach.\\

\noindent \textit{Related work:} {The penalty framework holds a key role in the realm of theoretical and numerical optimization, with its history tracing back at least to \cite{Cou:43}. Numerous works have  examined penalty methods for a wide range of  problems, see e.g.,  \cite{PolTre:73,  KonMel:18,IzmSol:23, Pol:87, NocWri:06, Fle:87, FiaMcC:68, CarGou:11,LinMa:22}. For example, in  \cite{IzmSol:23}  a general class of penalty functions of the form $\|\cdot\|_q^q$, with $q>0$,  is considered and bounds on  the distance of the solution of the penalty subproblem to the solution of the original problem are derived in terms of the penalty parameter $\rho$. In particular,  under strict Mangasarian–Fromovitz constraint qualification and second-order sufficiency, a bound of the form $\mathcal{O}({\rho^{1-q}})$ is derived and it becomes zero for $q\in(0,1]$, provided that $\rho$ is sufficiently large. In \cite{CarGou:11}, using a  Lipschitz penalty function  (e.g., Euclidean norm $\|\cdot\|$), a new algorithm is proposed based on the linearization of both, the objective function and the functional constraints, in the penalty subproblem and, additionally,  incorporating a dynamic quadratic regularization. It is shown that this algorithm requires at most $\mathcal{O}({\epsilon^{-2}})$ functions and first derivatives evaluations to obtain a reduction in the size of a first-order criticality measure below some accuracy $\epsilon$. In contrast to \cite{CarGou:11}, in  \cite{KonMel:18} a quadratic penalty method is used  when dealing with nonconvex composite programs subject to  linear constraints.  The study demonstrates that the algorithm, which involves sequentially minimizing the quadratic penalty function, converges to an $\epsilon$ stationary point when using the accelerated composite gradient (ACG) method for solving the subproblem. Notably, this convergence is achieved within $\mathcal{O}(\epsilon^{-3})$ ACG iterations.

\medskip 

\noindent  Furthermore,  in \cite{BirGar:16} a two-phase scaled penalty type algorithm is presented for general nonlinear programs. The first phase finds  an approximate feasible point, while the second phase involves generating a $p$ Taylor approximation of a quadratic penalty function and the next iterate is the solution that satisfies a decrease in this penalty  function. This algorithm requires  \(\mathcal{O}(\epsilon^{1-2(p+1)/p})\) number of evaluations of the problem's functions and derivatives up to order \(p\) to achieve an $\epsilon$  approximate first-order critical point for the original problem. It is noteworthy that when \(p=1\), no hessian information is used and the rate is \(\mathcal{O}(\epsilon^{-3})\).}  Finally, in \cite{LinMa:22} an   inexact proximal-point penalty method is proposed, where at each iteration  one solves inexactly  the quadratic penalty function  plus a proximal term. It is proved that when the objective function and constraints are nonconvex, this method converges to an $\epsilon$ first-order solution in $\mathcal{O}({\epsilon^{-3}})$ iterations, while for nonconvex objective  and convex functional constraints this rate it is improved to $\mathcal{O}({\epsilon^{-2.5}})$.

\medskip 
 
\noindent {Another promising approach when dealing with nonlinear constraints is the augmented Lagrangian (AL), which can be seen as a generalization of the quadratic penalty method. This approach has been extensively explored in the literature for nonlinear programs, as demonstrated in various works \cite{BerTsi:03, BirMar:14, BirMar:20, BoyPar:11, KreMar:00, XieWri:21, CohHal:21, Yas:22, SahEft:19}. In particular, in \cite{BirMar:14} an in-depth analysis of practical AL methods is given and connections between  quadratic penalty and AL methods are established.   Moreover, in the recent study \cite{XieWri:21}, a proximal augmented Lagrangian method is examined for nonconvex, yet smooth, optimization problems.  It is proved that when an approximate first- (second-) order solution of the subproblem is found inexactly, with an error approaching asymptotically zero, then an $\epsilon$ first- (second-) order solution of the original  problem is obtained within $\mathcal{O}(\epsilon^{\eta-2})$ outer iterations, where $\eta \in [0,2]$ is an user-defined parameter.  

\medskip 

\noindent Finally, Sequential Convex Programming (SCP) or inexact restoration techniques, offer an alternative to the penalty/augmented Lagrangian-based approaches  \cite{MesBau:21, TraDie:10,BueMar:20}. More, specifically, SCP solves a sequence of convex approximations of the original problem by linearizing the nonconvex parts of the objective and functional constraints while preserving structures that can be exploited by convex optimization techniques, see e.g., \cite{GraBoy:14,NecKva:15}. However, to the best of our knowledge,  SCP methods converge  under mild assumptions only  locally \cite{MesBau:21,TraDie:10}.  On the other hand, inexact restoration methods are based on a two phase procedure, first phase focuses on seeking a  feasible point, while the second phase aims to generate a point ensuring a decrease in the objective function.  In \cite{BueMar:20} it has been proven that such strategy  yields an approximate KKT point in  $\mathcal{O}\left({\epsilon^{-2}}\right)$ iterations. Nevertheless, each iteration requires solving a two-phase subproblem, which is usually computationally expensive.



\medskip 

\noindent \textit{Contributions:} The adopted approach, the linearized quadratic penalty method (LQP), addresses several key limitations present in the previous works. Notably, in \cite{CarGou:11}, the subproblem is nondifferentiable due to the adoption of a nonsmooth penalty function, and in \cite{KonMel:18}, only linear constraints can be handled. Furthermore, the Proximal AL method in \cite{XieWri:21} incurs high computational costs for solving nonconvex subproblems, and the SCP schemes \cite{MesBau:21, TraDie:10} provide only local convergence guarantees.  Specifically, our main  contributions are:

\begin{itemize}
\item[(i)] 
We linearize the cost function and the nonlinear functional  constraints in the quadratic penalty formulation at each iteration and we consider a {new adaptive}  regularization term. These lead to an algorithm, called the linearized quadratic penalty (LQP) method, which requires solving a simple unconstrained quadratic subproblem, making it easy to solve. Since our method uses only first-order information, it is a first-order type method, but with a better approximation model of the subproblem compared to, e.g., the gradient method \cite{BirGar:16}.

\item[(ii)]  We provide global convergence results for the LQP method. More specifically,  our method guarantees convergence to an $\epsilon$ first-order solution of the original problem in at most $\mathcal{O}({\epsilon^{-2.5}})$  iterations, thus improving the existing bounds. Moreover, under the  Kurdyka-Lojasiewicz (KL) property, we prove convergence of the entire sequence generated by LQP algorithm and derive improved local convergence rates that depend on the KL parameter.

\item[(iii)] Compared to \cite{CarGou:11}, which employs a Lipschitz penalty function and has a total complexity of order \(\mathcal{O}\left(\epsilon^{-3}\right)\)  {when employing smoothing and an accelerated gradient scheme for solving the convex nonsmooth subproblem,} our approach exhibits a total  complexity of order  \(\mathcal{O}\left(\epsilon^{-2.75} \log \left(\epsilon^{-1}\right)\right)\)
{when we employ an accelerated gradient method for solving the corresponding smooth (strongly) convex subproblem}, despite the complexity in evaluating Jacobians  being slightly more favorable for \cite{CarGou:11}. Still comparing the complexity of the subproblems, the algorithms in \cite{XieWri:21, LinMa:22} are difficult to implement in practice due to their high nonconvexity caused by the presence of nonlinear constraints in the subproblem from each iteration. Moreover, unlike \cite{KonMel:18}, our LQP method can handle general nonlinear equality constraints.  Finally, our LQP algorithm benefits from global convergence guarantees, which is an advantage over the SCP methods, for which only local convergence has been established.
\end{itemize}

\noindent These contributions collectively make  LQP method a powerful and versatile tool for efficiently solving smooth nonconvex problems with nonlinear equality constraints, surpassing some limitations of existing techniques. The performance of LQP is also validated through extensive numerical simulations. 

\medskip 

\noindent The paper is structured as follows. In Section \ref{sec2}, we introduce our problem of interest  and some notions necessary for our analysis. In Section \ref{sec3}, we present LQP algorithm, followed in Section \ref{sec4} by its convergence analysis. Finally, in Section \ref{sec5}, we compare numerically our method with existing algorithms. 


\section{Problem formulation and  preliminaries}
\label{sec2}
In this paper, we consider the following nonlinear optimization problem:
\begin{equation}
\begin{aligned}\label{eq1}
& \underset{x\in\mathbb{R}^n}{\min}
& & f(x)\\
& \hspace{0.2cm}\textrm{s.t.}
& & \hspace{0.07cm} F(x)=0,
\end{aligned}
\end{equation}
where $f:\mathbb{R}^{n}\to {\mathbb{R}}$ and  $F(x)\triangleq{(f_1(x),...,f_m(x))}^T$, with $f_i:\mathbb{R}^{n}\to {\mathbb{R}}$ for all $i=1:m$. We assume the functions  $f, f_i$  are continuously differentiable for all  $i=1:m$, $f$ is  nonconvex  and $F$ is nonlinear. {Moreover, we assume that the problem is well-posed, i.e., the feasible set is nonempty and the optimal value is finite. 
Before introducing the main assumptions for our analysis, we would like to clarify some notations.  We use $\|.\|$ to denote  the 2-norm of a vector or of a matrix, respectively.  For a differentiable function $f:\mathbb{R}^n\to\mathbb{R}$, we denote by $\nabla f(x)\in\mathbb{R}^n$ its gradient at a point $x$. For a differentiable vector function $F:\mathbb{R}^n  \to\mathbb{R}^m$, we denote its Jacobian at a given point $x$ by ${J}_F(x)\in\mathbb{R}^{m\times n}$. Let us now present the main assumptions considered  for problem \eqref{eq1}:

\begin{assumption}\label{assump1}
Assume that $f(x)$ has  compact level sets, i.e., for any $\alpha\in\mathbb{R}$, the following set is either empty or compact:
\[
\mathcal{S}_{\alpha}^0\triangleq{\{x:\, f(x)\leq\alpha\}}.
\]
\end{assumption}
\begin{assumption}\label{assump2}
Given a compact set $\mathcal{S}\subseteq\mathbb{R}^n$, there exist positive constants $ M_F, \sigma, L_f, L_F$ such that $f$ and $F$ satisfy the following conditions:
\begin{enumerate}[(i)]
  \item $ \|\nabla f(x)-\nabla f(y)\|\leq L_f\|x-y\| \;\text{ for all } x, y\in\mathcal{S}$.\label{ass1}
  \item $  \|J_F(x)\|\leq M_F, \hspace{0.5cm}\|J_F(x)-J_F(y)\|\leq L_F\|x-y\|\;\text{ for all } x, y\in\mathcal{S}$. \label{ass2}
  \item Problem  \eqref{eq1} satisfies Linear Independance Constraint Qualification (LICQ) condition  \text{ for all } $x\in\mathcal{S}$\label{ass3}.
\end{enumerate}
\end{assumption}

\begin{assumption}\label{assump3}
There exists finite $\bar{\alpha}$ such that $f(x)\leq\bar{\alpha}$ for all $x\in\{x:\, \|F(x)\|\leq1\}$.
\end{assumption}

\noindent Note that these assumptions are standard in the nonconvex optimization literature, in particular in penalty-type methods, see e.g., \cite{CarGou:11,XieWri:21,CohHal:21}. In fact, these assumptions, except LICQ (although it is also standard in nonconvex optimization), are not too restrictive as they need to hold only locally. Any twice-differentiable function is Lipschitz on a compact set, and its gradient is also Lipschitz on any compact set. Indeed, large classes of functions satisfy these assumptions, as discussed below.

\begin{remark}
Assumption \ref{assump1} holds e.g., when $f(x)$ is coercive ({in particular,} $f(x)$ is strongly convex),  or $f(x)$ is bounded {from} bellow. 
\end{remark}

\begin{remark}\label{remark_LICQ}
Assumption \ref{assump2} allows general classes of problems. In particular, condition \textit{(\ref{ass1})} holds if $f(x)$ is differentiable and $\nabla f(x)$ is \textit{locally}  Lipschitz continuous on a neighborhood of $\mathcal{S}$. Conditions \textit{(\ref{ass2})} hold when $F(x)$ is differentiable on a neighborhood of $\mathcal{S}$ and $J_F(x)$ is \textit{locally} Lipschitz continuous on $\mathcal{S}$. {Finally, the LICQ assumption guarantees the existence of dual multipliers and is commonly used in nonconvex optimization, see e.g.,  \cite{NocWri:06, XieWri:21}, as it implies that there exists $ \sigma > 0 $ such that $ \sigma_{\text{min}}(J_F(x)) \geq \sigma, \forall x \in \mathcal{S} $, where $\sigma_{\text{min}}(J_F(x))$ denotes the smallest singular value of $J_F(x)$. }
\end{remark}

\begin{remark}
For  Assumption \ref{assump3} to hold, it is sufficient the set $\{x: \,\|F(x)\|\leq 1\}$ to be compact. In fact, we do not need this assumption if we can choose the starting point $x_0$ such that $F(x_0)=0$, i.e., the initial point is feasible.
\end{remark}

\noindent The following lemma is an immediate consequence of Assumption \ref{assump1}.
 \begin{lemma} If  Assumption \ref{assump1} holds, then for any $\rho\geq 0$, we have:
\begin{equation}\label{lem1}
 \bar{L}\triangleq{\inf_{x\in\mathbb{R}^n}\{ f(x)+\frac{\rho}{2}\|F(x)\|^2\}}>-\infty \quad \text{ and } \quad     \bar{f}\triangleq{\inf_{x\in\mathbb{R}^n}\{ f(x) \}}>-\infty.   
\end{equation}
\end{lemma}

\noindent We are interested in (approximate) first-order (KKT) solution of \eqref{eq1}.  Hence,   let us  introduce the following definitions: 
\begin{definition}\label{firstorder}[First-order solution  and $\epsilon$ first-order solution of \eqref{eq1}]. 
The vector $x^*$ is a first-order (KKT) solution of \eqref{eq1} if  $\exists\lambda^*\in\mathbb{R}^m$ such that: 
\begin{equation*}
\nabla f(x^*)+J_F(x^*)^T\lambda^*=0\hspace{0.3cm} \text{and} \hspace{0.3cm} F(x^*)=0.
\end{equation*} 
Moreover, $x^*$ is an $\epsilon$ first-order solution of \eqref{eq1} if  exist $\lambda^*\in\mathbb{R}^m$ and two positive constants $c_1$ and $c_2$ such that:
\begin{equation*}
    \|\nabla f(x^*)+J_F(x^*)^T\lambda^*\|\leq c_1 \epsilon \hspace{0.3cm} \text{and} \hspace{0.3cm} \|F(x^*)\|\leq c_2 \epsilon.
\end{equation*} 
\end{definition}

\noindent Finally, let us  introduce the  Kurdyka-Lojasiewicz (KL) property, which leads to improvements in the convergence rates of LQP algorithm. In particular, its incorporation will prove useful in improving the global performance of LQP algorithm.  Let	 $\Phi: \mathbb{R}^d \to \bar{\mathbb{R}}$ be a proper lower semicontinuous function. For $−\infty < \tau_1 < \tau_2 \leq
 +\infty$, we define $[\tau_1<\Phi<\tau_2]=\{x \in\mathbb{R}^d :  tau_1<\Phi(x)<\tau_2\}$. 
\begin{definition} \label{def2}
Let $\Phi : \mathbb{R}^d \to \bar{\mathbb{R}}$  be a proper lower semicontinuous function that takes constant value on a set $\Omega$. We say that $\Phi$ satisfies the KL property on $\Omega$ if there exists $ \epsilon>0, \tau>0$ and $\varphi\in\Psi_{\tau}$ (where  $\Psi_{\tau}$ denotes the set of all continuous concave functions $\varphi: [0, \tau] \to [0,+\infty)$ satisfying $\varphi(0) = 0$ and $\varphi$ is continuously differentiable on $(0, \tau)$, with $\varphi' > 0$ over $(0, \tau)$) such that for every
$x^* \in \Omega$ and every element $x$ in the intersection $\{x\in\mathbb{R}^d: \text{ dist}(x,\Omega)<\epsilon\}\cap[\Psi(x^*)<\Psi(x)<\Psi(x^*)+\tau]$, we have:
\[
    \varphi'\big(\Phi(x) − \Phi(x^*)\big)\text{dist}\big(0, \partial\Phi(x)\big) \geq1.
\] 
\end{definition}
This definition covers many classes of functions arising in practical optimization problems. For example, if $f$ is a proper closed semialgebraic function, then $f$ is a KL function with exponent $\nu\in[0,1)$, see \cite{AttBol:13}. The function $g(Ax)$, where $g$ is strongly convex on a compact set and  twice differentiable, and $A \in \mathbb{R}^{m\times n}$, is
a KL function. Convex piecewise linear/quadratic functions such as $\|x\|_{1}, \|x\|_{0}, \gamma\sum_{i=1}^{k}|x_{[i]}|$, 
where $|x_{[i]}|$ is the $i$th largest  entry in $x, \;k\leq n$ and $\gamma \in (0, 1]$; the indicator function $\delta_{\Delta}(x)$, where $\Delta= \{x\in\mathbb{R}^n : e^T x = 1, x \geq 0\}$;  least-squares problems with the Smoothly Clipped Absolute Deviation (SCAD) \cite{Fan:97};  Minimax Concave Penalty (MCP)  regularized functions \cite{Zha:10} are all KL functions. 


\section{A linearized quadratic penalty method}\label{sec3}
In this section, we propose a new algorithm for solving nonconvex problem \eqref{eq1} using the quadratic penalty framework. Let us first introduce few  notations. The penalty function associated with the problem \eqref{eq1} is 
\begin{equation}\label{penalty_function}
    \mathcal{P}_{\rho}(x)=f(x)+\frac{\rho}{2}{\|F(x)\|^2}.
\end{equation}
The gradient of $\mathcal{P}_{\rho}$ is: $\nabla\mathcal{P}_{\rho}(x)={\nabla f(x)+J_F(x)}^T\left(\rho F(x)\right)$.  Further, let us denote the following function derived from linearization of  the objective function and the functional constraints, at a given point $\bar{x}$, in the penalty function:
\begin{align*}
 \bar{\mathcal{P}}_{\rho}(x;\bar{x})=f(\bar{x})+\langle\nabla f(\bar{x}),x-\bar{x}\rangle +\frac{\rho}{2}{\|F(\bar{x})+J_F(\bar{x})(x-\bar{x})\|^2}.   
\end{align*} 
\noindent Below, we also use the  following notations:
\begin{gather*}
     l_f(x;\bar{x}):=f(\bar{x})+\langle\nabla f(\bar{x}),x-\bar{x}\rangle,\quad    l_F(x;\bar{x}):=F(\bar{x})+ J_F(\bar{x})(x-\bar{x}) \hspace{0.5cm}\forall x,\bar{x}.
\end{gather*}

\noindent To solve the optimization problem \eqref{eq1} we propose the following \textit{Linearized Quadratic Penalty} (LQP) algorithm, i.e., we linearize the objective function and the functional constraints in the penalty function at the current iterate and add an \textit{adaptive} quadratic regularization.
\begin{algorithm}
\caption{Linearized quadratic penalty (LQP) method}\label{alg1}
\begin{algorithmic}[1]
\State  $\textbf{Initialization: } x_0, \ubar{\beta}>0 \text{ and } \rho>0$.
\State $k \gets 0$
\While{$\text{ stopping criterion is not satisfied }$}
    \State $\text{generate a proximal parameter } \beta_{k+1}\geq\ubar{\beta}$ such that
    \State $x_{k+1}\gets\argmin_{x\in\mathbb{R}^n}{\bar{\mathcal{P}}_{\rho}(x;x_{k})+\frac{\beta_{k+1}}{2}{\|x-x_{k}\|}^2}$ satisfies the following descent:
    \begin{equation}\label{decrease}
           \mathcal{P}_{\rho}(x_{k+1})\leq\mathcal{P}_{\rho}(x_{k})-\frac{\beta_{k+1}}{2}\|x_{k+1}-x_{k}\|^2.
    \end{equation}
    \State $k \gets k+1$
\EndWhile
\end{algorithmic}
\end{algorithm}

\noindent
Note that  the objective function in the subproblem of step 5 of  Algorithm \ref{alg1} is always unconstrained,  strongly convex and quadratic. Therefore, finding a solution of the subproblem in step 5 is equivalent to solving a linear system of equalities for which  efficient solution methods exist, i.e.:
\[x_{k+1}=x_k-(\rho J_F(x_k)^TJ_F(x_k)+\beta_{k+1}I_n)^{-1}(\nabla f(x_k)+\rho J_F(x_k)^TF(x_k)).\] 
Note our results still hold when the subproblem in Step 4 is solved inexactly.
\begin{remark}
   Note that, from the previous relation, when the objective function is null and $\rho=1$, the update of $x_{k+1}$ takes the form:
    \[x_{k+1} = x_k - (J_F(x_k)^T J_F(x_k) + \beta_{k+1}I_n)^{-1}(J_F(x_k)^T F(x_k)),\] 
    which corresponds  to the Levenberg-Marquardt method \cite{NocWri:06}. Moreover, when  $\beta_{k+1}=0$, we obtain the following update:
     \[x_{k+1} = x_k - ( J_F(x_k)^T J_F(x_k) )^{-1} J_F(x_k)^T F(x_k),\] 
     which yields the  Gauss-Newton method \cite{NocWri:06}.  Hence,  our  algorithm can be seen as a quasi-Newton method for minimizing the quadratic penalty function $\mathcal{P}_{\rho}$ defined in \eqref{penalty_function}, where the approximation of the Hessian at iteration $k$ is based on  first-order derivatives and is given by $\rho J_F(x_k)^T J_F(x_k) + \beta_{k+1}I_n$. This special approximation of the Hessian in our quasi-Newton type method allows us to derive global convergence results and rates under  general conditions and under the KL property.
\end{remark}

\noindent Let us show that we can always choose  {an adaptive regularization parameter} \(\beta_{k+1}\) guaranteeing  the descent property \eqref{decrease}. Indeed, since $f$ and $F$ are smooth functions, if one chooses:

\vspace{-0.3cm}
 
\begin{equation} \label{eq_assu}
    \beta_{k+1} \geq L_f + L_F\sqrt{2\rho}\sqrt{\mathcal{P}_{\rho}(x_k)-\bar{f}},
\end{equation} 
which depends on the current iterate $x_k$ then the  descent property \eqref{decrease} follows (hence, the right hand side in (5) decreases along the iterations of LQP algorithm) and this is proved in the following lemma.  In the rest of this paper, for the sake of clarity, we provide the proofs of all the lemmas in Appendix.
\begin{lemma}\label{lemma3}[Existence of $\beta_{k+1}$]  If  the sequence $\{x_{k}\}_{k\geq0}$ generated by Algorithm \ref{alg1} is in some compact set $\mathcal{S}$ on which Assumptions \ref{assump1} and \ref{assump2} hold and we choose $\beta_{k+1}$  as in \eqref{eq_assu}, then the descent property \eqref{decrease} is valid.
\end{lemma}
\begin{proof}
See Appendix.
\end{proof}
 \noindent Note that the choice of $\beta_{k+1}$ in \eqref{eq_assu} is inspired by \cite{MarOku:24} and the proof of the above lemma follows similar arguments as in \cite{MarOku:24}.  Lemma \ref{lemma3} proves that LQP algorithm is implementable, since choosing $\beta_{k+1}$ as in \eqref{eq_assu} already guarantees the descent in \eqref{decrease}. In practice,  the regularization parameter $\beta_k$ can be determined using a backtracking scheme see Algorithm \ref{backtracking}. 

\begin{algorithm}
\caption{Backtracking procedure}\label{backtracking}
\begin{algorithmic}[1]
\State  $\textbf{Initialization: } \mu>1 \text{ and } \beta^0_{k+1}=\max\{\beta_k/\mu, \underline{\beta}\}$.
\State $i \gets 0$
\While{$\text{ \eqref{decrease} is not satisfied  for $\beta_{k+1}=\beta^i_{k+1}$}$}
    \State $ \beta^{i+1}_{k+1}\gets\mu\beta^{i}_{k+1}$
    \State $x^{i+1}_{k+1}\gets\argmin_{x\in\mathbb{R}^n}{\bar{\mathcal{P}}_{\rho}(x;x_{k})+\frac{\beta^{i+1}_{k+1}}{2}{\|x-x_{k}\|}^2}$
    \State $i \gets i+1$
\EndWhile
\State $ \beta_{k+1}\gets\beta^{i+1}_{k+1}, \quad x_{k+1}\gets x^{i+1}_{k+1}$
\end{algorithmic}
\end{algorithm}
\color{black}

\noindent Specifically, at iteration $k$ we start with an initial value $\beta^0_{k+1} = \max\{\beta_k/\mu, \underline{\beta}\}$ for some $\mu>1$ and  generate the new iterate $x_{k+1}$ by solving the subproblem in Step 5 of Algorithm \ref{alg1}. Next, we test whether \eqref{decrease} is satisfied or not. If it is satisfied, we proceed to Step 6 of Algorithm \ref{alg1}. However, if \eqref{decrease} is not satisfied, we geometrically increase $\beta^0_{k+1}$ by a geometric factor (with inner values of the form $\beta^i_{k+1} := \mu^i \beta^0_{k+1}$, with $i \geq 0$), until \eqref{decrease} is satisfied. This process finishes in a finite number of steps, as proved in Lemma \ref{lemma3}. The value of $\beta^i_{k+1}$ for which \eqref{decrease} is satisfied, denoted  $\beta^{i_0}_{k+1}$, is then chosen as the value of $\beta_{k+1}$. 
}

\medskip

\noindent {There are good reasons for adopting a dynamic regularization $\beta_{k+1}$. One of them is that, in general, the parameters of the problem are usually unknown. Another reason is that, since $\beta_{k+1}$ is  {usually} generated based on the  {inequality \eqref{decrease}}, which must be tested only for the point $x_{k+1}$, this might results in  $\beta_k$'s that are much smaller than the bound in  \eqref{eq_assu}.} 
In the sequel we also denote: 
\[
\Delta x_{k}=x_{k}-x_{k-1} \hspace{0.3cm} \forall k\geq1.
\]

\section{Convergence analysis}\label{sec4}
In this section, we prove the effeciency of LQP algorithm (Algorithm \ref{alg1})  to obtain an $\epsilon$ first-order solution for  problem \eqref{eq1} and we derive  improved {local} convergence rates under the KL condition.  Through our analysis, we use a Lyapunov approach to prove the convergence of LQP algorithm.  
\noindent  Bellow, we are using the quadratic penalty function, $\mathcal{P}_{\rho}$,  as a Lyapunov function.  The evaluation of the Lyapunov function along the iterates of LQP algorithm  is denoted by:
\begin{equation}\label{lyapunov_function}
 P_{k}=\mathcal{P}_{\rho}(x_k) \hspace{0.3cm} \forall k\geq0.
\end{equation}

\noindent  It is clear from Lemma \ref{lemma3}  that $\{P_{k}\}_{k\geq0}$ is decreasing. 
In the sequel, we assume that $x_0$ is chosen such that:
\begin{equation}\label{eq4}
    \|F(x_0)\|^2\leq\min\left\{1,\frac{2c_0}{\rho}\right\} \hspace{0.7cm} \text{ for some } c_0>0,
\end{equation}
and that $f(x_0)\leq\bar{\alpha}$.
Using the definition of $\mathcal{P}_{\rho}$, we have: 
 \begin{align}
    \mathcal{P}_{\rho}(x_0)&=f(x_0)+\frac{\rho}{2}\|F(x_0)\|^2{\overset{{\eqref{eq4}}}{\leq}} \alpha := \bar{\alpha}+c_0 \label{ine10}.
\end{align} 

\noindent In fact, if  the sequence $\{x_{k}\}_{k\geq0}$ generated by Algorithm \ref{alg1} is in some compact set $\mathcal{S}$ on which Assumptions \ref{assump1}, \ref{assump2} and \ref{assump3} hold and, moreover,  $x_0$ is chosen as in \eqref{eq4} with $\rho\geq 1$, then we have the following bounds on $P_k$:
\begin{equation}\label{important}
    \bar{L}  \leq P_{k}\leq \alpha  \hspace{0.5cm}\forall k\geq0,
\end{equation}
where $\bar{L}$ is  defined in \eqref{lem1} and $\alpha$ as in \eqref{ine10}. 
\noindent Let us now bound the gradient of the penalty function.

\begin{lemma}\label{bounded_gradient}[Boundedness of $\nabla\mathcal{P}_{\rho}$]
If  the sequence $\{x_{k}\}_{k\geq0}$ generated by Algorithm \ref{alg1} is in some compact set $\mathcal{S}$ on which Assumptions \ref{assump1} and \ref{assump2} hold, then we have:
\begin{align}\label{bounde_gradient1}
    &\|\nabla\mathcal{P}_{\rho}(x_{k+1})\|  \nonumber \\
    & \leq \! \left( \! L_f \!+\! L_F \sqrt{2\rho} \sqrt{\mathcal{P}_{\rho}(x_k) \!-\! \bar{f}} \!+ \beta_{k+1} \! \right) \! \|\Delta x_{k+1}\| + \frac{\rho M_FL_F}{2} \|\Delta x_{k+1}\|^2.
\end{align}
Additionally, we can also get the following bound:
\begin{equation}\label{bounde_gradient}
    \|\nabla\mathcal{P}_{\rho}(x_{k+1})\| \!\leq\! \left( \! L_f \!+\! L_F \sqrt{2\rho} \sqrt{\mathcal{P}_{\rho}(x_k) \!-\! \bar{f}} \!+\! \beta_{k+1} \!+\! 2 \rho M_F^2  \! \right) \! \|\Delta x_{k+1}\|.
\end{equation}
\end{lemma}
\begin{proof}
See Appendix.
\end{proof}

\noindent From Lemma \ref{lemma3} it follows  that when using a backtracking scheme, such as Algorithm \ref{backtracking}, $\beta_{k+1}$ can be always upper bounded as follows:
\begin{equation} 
\label{bar_gamma}
    \bar{\beta}\triangleq \sup_{k\geq 1}\beta_k= \mu \left(L_f+ L_F\sqrt{2\rho}\sqrt{\alpha-\bar{f}}\right).
\end{equation}


\subsection{Global convergence rate} 
In this section a global convergence rate is given for the iterates generated by  Algorithm \ref{alg1} to an $\epsilon$ first-order (KKT) solution\footnote{The main difference between this paper and the journal version (Journal of Global Optimization, doi: 10.1007/s10898-024-01456-3, 2024) is that in the next convergence results  we do not require  the regularization parameter $\beta_{k+1}$ to  satisfy  \eqref{eq_assu}. This is useful when the problem's parameters $(L_f, L_F,\bar{f})$ are unknown. However, in this case we need to consider $\ubar{\beta} = \mathcal{O}\left(\sqrt{\rho}\right)$ in order to preserve the complexity  $\mathcal{O}(\epsilon^{-2.5})$.}. 

\begin{theorem}\label{main_result}[First-order complexity]
Consider Algorithm \ref{alg1} and let  $\{P_{k}\}_{k\geq0}$ be defined as in \eqref{lyapunov_function}. If  the sequence $\{x_{k}\}_{k\geq0}$ generated by Algorithm \ref{alg1} is in some compact set $\mathcal{S}$ on which Assumptions \ref{assump1}, \ref{assump2} and \ref{assump3} hold, then for any accuracy $\epsilon>0$,  choosing  $\rho = \max\{1, \mathcal{O}(\epsilon^{-1})\}, \ubar{\beta} = \mathcal{O}\left(\sqrt{\rho}\right)$  and $x_0$  as in \eqref{eq4}, Algorithm \ref{alg1} yields an $\epsilon$ first-order solution of \eqref{eq1} after $K=\mathcal{O}(\epsilon^{-2.5})$  iterations. Moreover, any limit point $x_{\rho}^*$ of the sequence $\{x_{k}\}_{k\geq0}$  satisfies:
\[
\nabla f(x_{\rho}^*) + {J_F(x_{\rho}^*)}^T \lambda_{\rho}^* = 0 \quad \text{ and } \quad \| F(x_{\rho}^*)\| = \frac{\|\lambda_{\rho}^*\|}{\rho} \leq \mathcal{O}(\epsilon).
\]
\end{theorem}
\begin{proof}
Let $k^*\geq 0$ be the first integer such that:
\begin{equation}
    \|\nabla\mathcal{P}_{\rho}(x_{k^*+1})\|\leq \epsilon.
\end{equation}
Using  \eqref{bounde_gradient1} from Lemma \ref{bounded_gradient}, it follows that at each iteration, $k \in [0:k^*]$, we encounter one of the following two cases:\\
\noindent \textit{Case 1: } If the following holds: $${\left(L_f+ L_F \sqrt{2\rho} \sqrt{P_k - \bar{f}} + \beta_{k+1}\right)\|\Delta x_{k+1}\|\geq \frac{\rho M_F L_F}{2}\|\Delta x_{k+1}\|^2},$$ then we get:
\begin{align} \label{case1}
& \|\nabla\mathcal{P}_{\rho}(x_{k+1})\|^2 \overset{\eqref{important}}{\leq} 2 \left(L_f+ L_F \sqrt{2\rho} \sqrt{\alpha - \bar{f}} + \beta_{k+1}\right)^2 \|\Delta x_{k+1}\|^2 \nonumber\\
&\overset{\eqref{decrease}}{\leq} \frac{4\left(L_f+ L_F \sqrt{2\rho} \sqrt{\alpha - \bar{f}} + \beta_{k+1}\right)^2}{\beta_{k+1}}\left(\mathcal{P}_{\rho}(x_{k}) - \mathcal{P}_{\rho}(x_{k+1})\right).
\end{align}
\noindent \textit{Case 2: } Otherwise,  the following is valid: 
$$\left(L_f+ L_F \sqrt{2\rho} \sqrt{P_k - \bar{f}} + \beta_{k+1}\right)\|\Delta x_{k+1}\|< \frac{\rho M_F L_F}{2}\|\Delta x_{k+1}\|^2,$$ 
which yields:
\begin{equation}\label{case2}
\|\nabla\mathcal{P}_{\rho}(x_{k+1})\|\leq \rho M_FL_F \|\Delta x_{k+1}\|^2\overset{\eqref{decrease}}{\leq} \frac{2\rho M_FL_F}{\beta_{k+1}}\left(\mathcal{P}_{\rho}(x_{k}) \!-\! \mathcal{P}_{\rho}(x_{k+1})\right).
\end{equation}
Define $\mathcal{I}_1$ as the set of iterations $k \in [0:k^*-1]$ at which   \textit{Case 1} holds, and  $\mathcal{I}_2$ as the set of iterations $k\in [0:k^*-1]$ at which   \textit{Case 2} holds. Clearly: $k^*=|\mathcal{I}_1|+|\mathcal{I}_2|+1$. We first derive an upper bound for $|\mathcal{I}_1|$. Summing \eqref{case1} over $\mathcal{I}_1$ yields:
\begin{align*}
 |\mathcal{I}_1|\epsilon^2  & < \sum_{k\in\mathcal{I}_1} \|\nabla\mathcal{P}_{\rho}(x_{k+1})\|^2  \\
& \leq  \sum_{k\in\mathcal{I}_1} \frac{4\left(L_f+ L_F \sqrt{2\rho} \sqrt{\alpha - \bar{f}} + \beta_{k+1}\right)^2}{\beta_{k+1}}\left(P_k - P_{k+1}\right)\\
   & \leq \frac{4\left(L_f+ L_F \sqrt{2\rho} \sqrt{\alpha - \bar{f}} + \bar{\beta}\right)^2}{\ubar{\beta}} \sum_{k=0}^{k^*\!-1} \left(P_k - P_{k+1}\right)  \\
   &\leq \frac{4\left(L_f+ L_F \sqrt{2\rho} \sqrt{\alpha - \bar{f}} + \bar{\beta}\right)^2}{\ubar{\beta}} \left(P_0 - P_{k^*}\right)  \\ 
   & {\overset{{\eqref{important} }}{\leq}}\frac{4\left(L_f+ L_F \sqrt{2\rho} \sqrt{\alpha - \bar{f}} + \bar{\beta}\right)^2}{\ubar{\beta}} \left(\alpha-\bar{L}\right).
\end{align*}
Thus, we have: ${|\mathcal{I}_1|< \frac{4\left(L_f+ L_F \sqrt{2\rho} \sqrt{\alpha - \bar{f}} + \bar{\beta}\right)^2}{\ubar{\beta}\epsilon^2} \left(\alpha-\bar{L}\right)}$. Similarly, we derive an upper bound for $|\mathcal{I}_2|$. Summing  \eqref{case2} over $\mathcal{I}_2$ yields:
\begin{align*}
 &  |\mathcal{I}_2|\epsilon  < \sum_{k\in\mathcal{I}_2} \|\nabla\mathcal{P}_{\rho}(x_{k+1})\|^2  \leq  \sum_{k\in\mathcal{I}_2} \frac{2\rho M_FL_F}{\beta_{k+1}}\left(P_k - P_{k+1}\right)\\
   & \leq  \sum_{k=0}^{k^*\!-1} \frac{2\rho M_FL_F}{\beta_{k+1}}\left(P_k - P_{k+1}\right)  \leq \frac{2\rho M_FL_F}{\ubar{\beta}} \left(P_0 - P_{k^*}\right){\overset{{\eqref{important} }}{\leq}}\frac{2\rho M_FL_F}{\ubar{\beta}} \left(\alpha-\bar{L}\right).
\end{align*}
Therefore, we obtain: $|\mathcal{I}_2|< \frac{2\rho M_FL_F}{\ubar{\beta}\epsilon} \left(\alpha-\bar{L}\right)$.  Consequently, we have:
\[
k^* \leq \left(\alpha-\bar{L}\right) \left({\frac{4\left(L_f+ L_F \sqrt{2\rho} \sqrt{\alpha - \bar{f}} + \bar{\beta}\right)^2}{\ubar{\beta}\epsilon^2}} + \frac{2\rho M_FL_F}{\ubar{\beta}\epsilon} \right).
\]
Choosing the penalty parameter  $\rho=\mathcal{O}(\epsilon^{-1})$, from \eqref{bar_gamma} we also have that $\bar{\beta}=\mathcal{O}(\sqrt{\rho}) = \mathcal{O}(\sqrt{\epsilon^{-1}})$ {and choosing $\ubar{\beta}= \mathcal{O}(\sqrt{\rho}) = \mathcal{O}(\sqrt{\epsilon^{-1}})$}, it follows that $k^*$ is of order {$k^*=\mathcal{O}\left(\frac{\epsilon^{-1}}{\sqrt{\epsilon^{-1}}\epsilon^2} + \frac{\epsilon^{-1}}{\epsilon}\right) = \mathcal{O}\left(\frac{1}{\epsilon^{2.5}} +  \frac{1}{\epsilon^2} \right)$}.  Therefore,  we have $\|\nabla\mathcal{P}_{\rho}(x_{k^*+1})\|\leq\epsilon$ within $k^* =\mathcal{O}(\frac{1}{\epsilon^{2.5}}) $ iterations, i.e.: 
\begin{align}\label{KKT1}
    \|\nabla f(x_{k^*+1})+J_F(x_{k^*+1})^T\lambda_{k^*+1}\|\leq\epsilon, \text{ where } \lambda_{k^*+1}=\rho F(x_{k^*+1}).
\end{align}
Below,  we show that $\lambda_{k^*+1}$ is bounded.  Indeed, since  LICQ holds on $\mathcal{S}$ and since $x_{k^*+1} \in \mathcal{S}$, it follows from Remark \ref{remark_LICQ} that
there exists $ \sigma > 0 $ such that $\sigma \|F(x_{k^*+1})\|\leq \|J_F(x_{k^*+1})^TF(x_{k^*+1})\| $. Then, using  the triangle inequality in \eqref{KKT1}, we obtain that:
\[
\sigma\|\lambda_{k^*+1}\|=\sigma\|\rho F(x_{k^*+1})\|\leq\rho\|J_F(x_{k^*+1})^TF(x_{k^*+1})\|\leq\|\nabla f(x_{k^*+1})\|+\epsilon.
\]
Given the continuity of  $\nabla f$, along with the fact that $x_{k^*+1}$ belongs to the compact set $\mathcal{S}$, we conclude that there exists a constant $M\geq 0$ such that
$\|\lambda_{k^*+1}\|\leq \frac{M}{\sigma}.
$
 Hence, provided that $\rho\geq\mathcal{O}(\epsilon^{-1})$, we get:
\[
\|F(x_{k^*+1})\|=\frac{\|\lambda_{k^*+1}\|}{\rho}\leq\mathcal{O}(\epsilon).
\]
\noindent It follows that $x_{k^*+1}$ is an $\epsilon$ first-order solution of problem \eqref{eq1} and this is achieved after  $k^*=\mathcal{O}(\frac{1}{\epsilon^{2.5}})$  iterations. 
Moreover, from Lemma \ref{lemma3}, we have:
\begin{align*}
     \frac{\beta_{k+1}}{2}\|&\Delta x_{k+1}\|^2\leq P_{k}-P_{k+1}\hspace{0.3cm}\forall k\geq0.
\end{align*}
By summing up the above inequality from $i=0$ to $i=k$, we obtain:
\begin{align}
\sum_{i=1}^{\infty}{\frac{\beta_{i+1}}{2}\|\Delta x_{i+1}\|^2} \leq \sum_{i=0}^{k}{\frac{\beta_{i+1}}{2}\|\Delta x_{i+1}\|^2}&\leq P_{0}-P_{k+1}{\overset{\eqref{important}}{\leq}}\alpha-\bar{L} < \infty.\label{limit}
\end{align}
This, together with  $\beta_{k}\geq\ubar{\beta}>0$, yields that $\lim_{k\to\infty}{\|\Delta x_{k}\|}=0$. Since the sequence $\{x_{k}\}_{k\geq0}$ is bounded, then  there exists a convergent subsequence, say  $\{x_{k}\}_{k\in\mathcal{K}}$, with the limit $x_{\rho}^*$.
Combining  \eqref{bounde_gradient1} from Lemma \ref{bounded_gradient}  with \eqref{bar_gamma}, we have:
\[
\|\nabla\mathcal{P}_{\rho}(x_{\rho}^*)\|=\lim_{k\in\mathcal{K}}{\|\nabla\mathcal{P}_{\rho}(x_{k})\|}\leq 2\bar{\beta}\lim_{k\in\mathcal{K}}\|\Delta x_{k+1}\| + \frac{\rho M_FL_F}{2} \lim_{k\in\mathcal{K}}\|\Delta x_{k+1}\|^2=0.
\]
Therefore, $ \nabla\mathcal{P}_{\rho}(x_{\rho}^*)=0$, i.e., 
\[
\nabla f(x_{\rho}^*) + {J_F(x_{\rho}^*)}^T \lambda_{\rho}^* = 0 \quad \text{ and } \quad \| F(x_{\rho}^*)\| = \frac{\|\lambda_{\rho}^*\|}{\rho} \leq \mathcal{O}(\epsilon).
\]
This concludes our proof. \qed
\end{proof}
\noindent From previous theorem, one can see that  in addition to its straightforward implementation, LQP algorithm also enjoys  global  convergence results, which gives it an advantage over approaches where only local convergence can be guaranteed, such as  SCP schemes \cite{MesBau:21}. One of its key advantages lies in its avoidance of calling complicated subroutines, as the subproblem is smooth strongly  convex quadratic and unconstrained, making it remarkably efficient compared to \cite{CarGou:11}, where the subproblem is not smooth, and compared to \cite{XieWri:21, LinMa:22}, where the subproblem is nonconvex.  In fact, solving the subproblem in our approach simplifies to solving a linear system of equalities.  With global theoretical convergence to an  $\epsilon$ first-order solution in at most $\mathcal{O}(\frac{1}{\epsilon^{2.5}})$ iterations, our method ensures the reliable finding of optimal solutions for a broad spectrum of nonconvex optimization problems, even those with nonlinear constraints, unlike \cite{KonMel:18} which  can only handle linear equality constraints.  Hence, its simplicity and effectiveness  make it an attractive algorithm for a wide range of practical applications.

\subsection{Improved rates under KL}
In this section, under the KL condition,  we provide better convergence rates for the iterates of Algorithm \ref{alg1}. In particular, we prove that the whole sequence  $\{x_{k}\}_{k\geq0}$ converges. Let's first derive several inequalities that will be useful for future proofs. From  \eqref{bounde_gradient}, we have the following:
\begin{equation}\label{key_formul}
\|\nabla\mathcal{P}_{\rho}(x_{k+1})\|^2\leq\Gamma_{\text{max}}^2\|\Delta x_{k+1}\|^2, \quad \text{where }   \; \Gamma_{\text{max}} := 2(\bar{\beta} + \rho M_F^2).
\end{equation}
Then, it follows from \eqref{key_formul} and Lemma \ref{lemma3}, that:
 \begin{equation}\label{rate}
        P_{k+1}-P_{k}\leq-\frac{\ubar{\beta}}{2\Gamma_{\text{max}}^2}\left\|\nabla\mathcal{P}_{\rho}(x_{k+1})\right\|^2.
 \end{equation}
\noindent We denote by $\text{crit } \mathcal{P}_{\rho}$ the set of critical points of the function $\mathcal{P}_{\rho}(\cdot)$ defined in \eqref{lyapunov_function}. Furthermore, we denote $\mathcal{E}_{k}=P_{k}-P^{*}$, where $P^{*}=\lim_{k\to\infty}{P_{k}}$ (recall that  the sequence $\{P_k\}_{k\geq0}$ is decreasing and bounded from bellow according to  Lemma \ref{lemma3} and \eqref{important}, respectively, hence it is convergent). Let us denote the set of limit points of $\{x_k\}_{k\geq0}$ by:
\[
\Omega:=\{x_{\rho}^*\;:\; \exists \text{ a convergent subsequence} \;  \{x_k\}_{k\in\mathcal{K}} \; \text{such that} \lim_{k\in\mathcal{K}}{x_k}=x_{\rho}^*\}.
\]
Let's now introduce the following lemma, which highlights the relationship between the set of limit points of the generated sequence, denoted $\Omega$, and the critical points of the Lyapunov function, denoted by  $\text{crit} \,  \mathcal{P}_{\rho}$.
\begin{lemma}\label{added_lemma}
Consider Algorithm \ref{alg1}  and let $\{P_{k}\}_{k\geq0}$ be defined as in \eqref{lyapunov_function}. If  the sequence, $\{x_{k}\}_{k\geq0}$, generated by Algorithm \ref{alg1} is in some compact set $\mathcal{S}$ on which Assumptions \ref{assump1}, \ref{assump2} and \ref{assump3} hold and, moreover,  $x_0$ is chosen as in \eqref{eq4} with $\rho\geq 1$, then the following statements hold:
\begin{enumerate}[(i)]
  \item $\Omega$  is a compact subset of \text{crit} $\mathcal{P}_{\rho}$ and   $ \lim_{k\to\infty}{\text{dist}(x_k,\Omega)}=0$.\label{lem_item1}
     \item For any $x\in\Omega,$ we have $\mathcal{P}_{\rho}(x)=P^{*}$.\label{lem_item2}
\end{enumerate}
\end{lemma}
\begin{proof}
    See Appendix.
\end{proof}

\noindent  Let us now show that the sequence $\|\Delta x_{k}\|$ has  finite length, provided that a KL condition holds.
\begin{lemma}\label{finite_length}
If  the sequence, $\{x_{k}\}_{k\geq0}$ generated by Algorithm \ref{alg1} is in some compact set $\mathcal{S}$ on which Assumptions \ref{assump1}, \ref{assump2} and \ref{assump3} hold and assume that $ \mathcal{P}_{\rho}(\cdot)$ satisfies the KL property on $\Omega$, then $\{x_{k}\}_{k\geq0}$ satisfies the finite length property, i.e.,
\[
\sum_{k=1}^{\infty}{\|\Delta x_{k}\|}<\infty,
\]
and consequently converges to a critical point of $\mathcal{P}_{\rho}(\cdot)$.
\end{lemma}
\begin{proof}
See Appendix.
\end{proof}
 
\noindent  
Lemma \ref{finite_length} shows that under the KL property  the set of limit points of the sequence $\{x_{k}\}_{k\geq0}$ generated by Algorithm \ref{alg1} is a singleton, denoted by $x_{\rho}^*$. In the next theorem, we will prove how fast the iterates converge to this limit point (see also \cite{Yas:22} for a similar reasoning).

 \begin{theorem}\label{main_result2}[Convergence of the whole sequence $\{x_{k}\}_{k\geq0}$] Let  Assumptions \ref{assump1}, \ref{assump2} and \ref{assump3} hold and  {$\mathcal{P}_{\rho}(\cdot)$ satisfy the KL property on $\Omega = \{x_{\rho}^*\}$, where 
 $x_{\rho}^*$ is the limit point of the sequence $\{x_k\}_{k\geq0}$ generated by Algorithm \ref{alg1}.  Then,} there exists  $k_1\geq1$ such that for all $k\geq k_1$ we have:
       \begin{equation*}\label{rate_point}
           \|x_{k}-x_{\rho}^*\|\leq C\max\{\varphi(\mathcal{E}_{k}),\sqrt{\mathcal{E}_{k-1}}\},
       \end{equation*}
       where $C>0$  and  $\varphi\in\Psi_{\tau}$, with $\tau>0$, denotes a desingularizing function. 
 \end{theorem}
 
 \begin{proof}
     From Lemma \ref{lemma3}, we get:
\begin{align}\label{llyap_thm}
     P_{k+1}-P_{k}{\overset{{}}{\leq}}-\frac{\ubar{\beta}}{2}\|x_{k+1}-x_{k}\|^2.
\end{align}
 Based on our choice of $\rho$ and $\beta_k$, the sequence $\{P_k\}_{k\geq0}$ is monotonically decreasing, see Lemma \ref{lemma3}, and consequently $\{\mathcal{E}_{k}\}_{k\geq0}$ is  monotonically decreasing. 
 Using \eqref{llyap_thm} and the fact that  $\{\mathcal{E}_{k}\}_{k\geq0}$ is nonnegative, we have for all $k\geq0$: 
\begin{equation}\label{lmit2}
    \|\Delta x_{k+1}\|\leq \sqrt{\frac{2}{\ubar{\beta}}}\sqrt{\mathcal{E}_{k}}.
\end{equation}
Since $ P_{k}\to P^{*}$, ${x_k}\to x_{\rho}^*$ and $\mathcal{P}_{\rho}(\cdot)$ satisfies the KL property at $x_{\rho}^*$, then there exists   $k_1=k_1(\epsilon,\tau)\geq 1$   such that $\forall k>k_1$, we have $\|x_k-x_{\rho}^*\|\leq \epsilon$ and $P^{*}<P_{k}<P^{*}+\tau$, and  the following KL property holds:
 \begin{equation}\label{KL1}
     \varphi'(\mathcal{E}_{k})\|\nabla\mathcal{P}_{\rho}(x_{k})\|\geq1.
 \end{equation}
 Since $\varphi$ is concave function, we have $\varphi(\mathcal{E}_{k})-\varphi(\mathcal{E}_{k+1})\geq\varphi'(\mathcal{E}_{k})(\mathcal{E}_{k}-\mathcal{E}_{k+1})$. Therefore, from \eqref{llyap_thm} and \eqref{KL1}, we get:
 \begin{align*}
 \|x_{k+1}-x_{k}\|^2&\leq\varphi'(\mathcal{E}_{k})\|x_{k+1}-x_{k}\|^2\|\nabla\mathcal{P}_{\rho}(x_{k})|\nonumber\\
 &\leq\frac{2}{\ubar{\beta}}\varphi'(\mathcal{E}_{k})(\mathcal{E}_{k}-\mathcal{E}_{k+1})\|\nabla\mathcal{P}_{\rho}(x_{k})\|\nonumber\\
 &\leq\frac{2}{\ubar{\beta}}\Big(\varphi(\mathcal{E}_{k})-\varphi(\mathcal{E}_{k+1})\Big)\|\nabla\mathcal{P}_{\rho}(x_{k})\|.
 \end{align*}
 Note that for given $a,b,c\geq0$, if we have $ {a^2}\leq2 b\times c$, and recognizing that we always have $2 b\times c\leq(b+c)^2$, then  $ {a^2}\leq (b+c)^2$, which in turn implies $a\leq b+c$.  It follows that for any $\theta>0$, we have (by taking $a=\|\Delta x_{k+1}\|$, $b=\frac{\theta}{\ubar{\beta}}\Big(\varphi(\mathcal{E}k)-\varphi(\mathcal{E}{k+1})\Big)$, and $c=\frac{1}{\theta}\|\nabla\mathcal{P}_{\rho}(x_k)\|$):
\begin{align}\label{lmit1}
    {\|\Delta x_{k+1}\|}\leq&\frac{\theta}{\ubar{\beta}}\Big(\varphi(\mathcal{E}_{k})-\varphi(\mathcal{E}_{k+1})\Big)+\frac{1}{\theta}\|\nabla\mathcal{P}_{\rho}(x_{k})\|.
\end{align}
Furthermore, from  \eqref{bounde_gradient}, we have:
\begin{align*}
    \| \nabla\mathcal{P}_{\rho}(x_{k})\|{\overset{{}}{\leq}}\Gamma_{\text{max}}\|\Delta x_{k}\|.
\end{align*}
 Then, \eqref{lmit1} becomes:
\begin{align*}
    \|\Delta x_{k+1}\|\leq&\frac{\theta}{\ubar{\beta}}\Big(\varphi(\mathcal{E}_{k})-\varphi(\mathcal{E}_{k+1})\Big)+\frac{\Gamma_{\text{max}}}{\theta}\|\Delta x_{k}\|.
\end{align*}
It follows that:
\begin{align*}
   \left(1-\frac{\Gamma_{\text{max}}}{\theta}\right) \|\Delta x_{k+1}\|\leq&\frac{\theta}{\ubar{\beta}}\Big(\varphi(\mathcal{E}_{k})-\varphi(\mathcal{E}_{k+1})\Big)+\frac{\Gamma_{\text{max}}}{\theta}\Big(\|\Delta x_{k}\|-\|\Delta x_{k+1}\|\Big).
\end{align*} 
By summing up the above inequality over $k\geq k_1$, we get: 
\begin{align*}
\left(1-\frac{\Gamma_{\text{max}}}{\theta}\right)\sum_{k\geq k_1}{\|\Delta x_{k+1}\|}\leq&\frac{\theta}{\ubar{\beta}}\varphi(\mathcal{E}_{{k_1}})+\frac{\Gamma_{\text{max}}}{\theta}\|\Delta x_{k_1}\|.
\end{align*}
Let us now choose $\theta>0$ such that $0<\frac{\Gamma_{\text{max}}}{\theta}<1$ and define  $\delta_0$ as  $\delta_0=1-\frac{\Gamma_{\text{max}}}{\theta}>0$. Then,  we have: 
\begin{align*}
   \sum_{k\geq k_1}{\|\Delta x_{k+1}\|}\leq&\frac{\theta}{\ubar{\beta}\delta_0}\varphi(\mathcal{E}_{{k_1}})+\frac{\Gamma_{\text{max}}}{\theta\delta_0}\|\Delta x_{k_1}\|.
\end{align*}
Hence, using the triangle inequality, we get for any $k \geq k_1$:
 \begin{align*}
  \|x_{k}-x_{\rho}^*\| \leq \sum_{l\geq k}{\|x_{l}-x_{l+1}\|}
 \leq\frac{\theta}{\ubar{\beta}\delta_0}\varphi(\mathcal{E}_{{k}})  +\frac{\Gamma_{\text{max}}}{\theta\delta_0}\|\Delta x_{k}\|.
 \end{align*}
 Further, using \eqref{lmit2}, it follows that:
  \begin{align*}
 \|x_{k}-x_{\rho}^*\|
 \leq\frac{\theta}{\ubar{\beta}\delta_0}\varphi(\mathcal{E}_{{k}})+\frac{\Gamma_{\text{max}}}{\theta\delta_0}\sqrt{\frac{2}{\ubar{\beta}}}\sqrt{\mathcal{E}_{k-1}}\leq C \max\{\varphi(\mathcal{E}_{{k}}),\sqrt{\mathcal{E}_{k-1}}\},
 \end{align*}
 where 
$
     C= \frac{\theta}{\ubar{\beta}\delta_0}+\frac{\Gamma_{\text{max}}}{\theta\delta_0}\sqrt{\frac{2}{\ubar{\beta}}}.
$
 This concludes our proof.\qed
 \end{proof}
{\color{black}
 \noindent The next theorem derives  {local} convergence rates  {for} the sequence generated by Algorithm \ref{alg1} when the Lyapunov function satisfies the KL property for a specific desingularizing function.
 \begin{theorem}\label{corollary} [Convergence rates of $\{x_{k}\}_{k\geq0}$] Let the assumptions of Theorem \ref{main_result2} hold and   
 $x_{\rho}^*$ be the limit point of the sequence $\{x_{k}\}_{k\geq0}$ generated by Algorithm \ref{alg1}. If $\mathcal{P}_{\rho}(\cdot)$ satisfies KL property at $x_{\rho}^*$ with  desingularizing function
       \[
       \varphi:[0,\tau)\to[0,+\infty),\; \varphi(s)=s^{1-\nu}, \text{ where } \nu\in[0,1),
       \]
       then the following rates hold:
       \begin{enumerate}
         \item If $\nu=0$, then $x_{k}$ converges to $x_{\rho}^*$ in a finite number of iterations.
     \item If $\nu\in(0,\frac{1}{2}]$, then there exists $k_1$ such that for all $k\geq k_1 $, we have the following linear convergence:
         \[
         \|x_{k}-x_{\rho}^*\|\leq C\frac{\sqrt{\mathcal{E}_{k_1}}}{\sqrt{(1+\bar{c}\mathcal{E}_{k_1}^{2\nu-1}})^{k-k_1}}, \quad \text{where} \;\; \bar{c}=\frac{\ubar{\beta}}{(1-\nu)^2\Gamma_{\text{max}}^2}.
         \]
         
         \item If $\nu\in(\frac{1}{2},1)$, then there exists $k_1$ such that for all $k> k_1 $, we have the following sublinear convergence:
         \[
         \|x_{k}-x_{\rho}^*\|\leq C\left(\frac{1}{\mu(k-k_1)+\mathcal{E}_{k_1}^{1-2\nu}}\right)^{\frac{1-\nu}{2\nu-1}}.
         \]
       \end{enumerate}
 \end{theorem}
 }
 \begin{proof}
 Let $ \nu\in[0,1)$ and for all  $ s\in [0,\tau), \varphi(s)=s^{1-\nu}$ and  $\varphi'(s)=(1-\nu)s^{-\nu}$.  From Theorem \ref{main_result2}, it follows that there exists $k_1\geq 1$ such that $\forall k\geq k_1$, we have:
        \begin{equation}\label{rate_point1}
           \|x_{k}-x_{\rho}^*\|\leq C\max\{\mathcal{E}_{k}^{1-\nu},\sqrt{\mathcal{E}_{k-1}}\}.
       \end{equation}
Furthermore,  \eqref{KL1} yields:
       \[ {\mathcal{E}_k}^{\nu}\leq (1-\nu)\|\nabla\mathcal{P}_{\rho}(x_{k})\| \hspace{0.5cm} \forall k\geq k_1.\]
Moreover, from \eqref{rate}, we have for any $k\geq 1$:
 \[
        \|\nabla\mathcal{P}_{\rho}(x_{k})\|^2\leq\frac{2\Gamma_{\text{max}}^2}{\ubar{\beta}}(\mathcal{E}_{k-1}-\mathcal{E}_{k}).
 \]
 Hence, 
$    {\mathcal{E}_k}^{2\nu}\leq \frac{2(1-\nu)^2\Gamma_{\text{max}}^2}{\ubar{\beta}}(\mathcal{E}_{k-1}-\mathcal{E}_{k}) \hspace{0.5cm} \forall k> k_1.$
Setting $\bar{c}=\frac{\ubar{\beta}}{2(1-\nu)^2\Gamma_{\text{max}}^2}>0,$
 we get:  
 \begin{equation} \label{c-bar}
     \bar{c}{\mathcal{E}_k}^{2\nu}\leq\mathcal{E}_{k-1}-\mathcal{E}_{k} \hspace{0.5cm} \forall k> k_1.
 \end{equation} 
 \begin{enumerate}
         \item Let $\nu=0$. If $\mathcal{E}_k>0$ for any $k> k_1$, we have $\bar{c}\leq \mathcal{E}_{k-1}-\mathcal{E}_{k}$. As $k$ goes to infinity, the right hand side approaches zero. Then, $0<\bar{c}\leq0$ which is a contradiction. Hence,  there exists $ k> k_1 $ such that $ \mathcal{E}_k=0.$ Then, $\mathcal{E}_k\to 0$ in a finite number of steps and  from \eqref{rate_point1}, $x_k\to x_{\rho}^*$ in a finite number of steps.
         \item Let $\nu\in(0,\frac{1}{2}]$. Then, $2\nu-1 \leq 0.$
         Let $k > k_1$. Since $\{\mathcal{E}_i\}_{i\geq k_1}$ is monotonically decreasing, then $\mathcal{E}_i\leq\mathcal{E}_{k_1}$ for any $i\in\{k_1+1, k_1+2,..., k\}$ and 
         \[\bar{c}{\mathcal{E}_{k_1}}^{2\nu-1}\mathcal{E}_k\leq\mathcal{E}_{k-1}-\mathcal{E}_{k} \hspace{0.5cm} \forall k> k_1.\]
         Rearranging this, we get  for all $k> k_1$:
       \[ 
       \mathcal{E}_k\leq \frac{\mathcal{E}_{k-1}}{1+\bar{c}{\mathcal{E}_{k_1}}^{2\nu-1}}\leq\frac{\mathcal{E}_{k-2}}{(1+\bar{c}{\mathcal{E}_{k_1}}^{2\nu-1})^2}\leq...\leq\frac{\mathcal{E}_{k_1}}{(1+\bar{c}{\mathcal{E}_{k_1}}^{2\nu-1})^{k-k_1}}.
       \]
  Since $\{\mathcal{E}_k\}_{k\geq k_1}$ is monotonically decreasing, then $\mathcal{E}_{k}\leq\mathcal{E}_{k-1}$, and since $\nu\in(0,\frac{1}{2}]$, then $1-\nu\geq\frac{1}{2}$. Hence,
\[
\mathcal{E}_{k}^{1-\nu}\leq \mathcal{E}_{k}^{\frac{1}{2}}\leq \mathcal{E}_{k-1}^{\frac{1}{2}}.
\]
Then, it follows that:
\[
 \max\{\mathcal{E}_k^{1-\nu}, \sqrt{\mathcal{E}_{k-1}} \}= \sqrt{\mathcal{E}_{k -1}}.
\]  
       Hence, we get:
        \[
         \|x_{k}-x_{\rho}^*\|\leq C\frac{\sqrt{\mathcal{E}_{k_1}}}{\sqrt{(1+\bar{c}\mathcal{E}_{k_1}^{2\nu-1}})^{k-k_1}},
         \]
  \item Let $\nu\in(\frac{1}{2},1)$. From \eqref{c-bar}, we have: 
       \begin{equation}\label{eqqq}
           \bar{c}\leq(\mathcal{E}_{k-1}-\mathcal{E}_k){\mathcal{E}^{-2\nu}_k} \hspace{0.5cm} \forall k> k_1.
       \end{equation}
    Let $h:\mathbb{R}^{*}_{+}\to\mathbb{R}$ be defined as $h(s)=s^{-2\nu}$ for any $s\in\mathbb{R}^{*}_{+}$. It is clear that $h$ is monotonically decreasing and $\forall s\in\mathbb{R}_+, h'(s)=-2\nu s^{-(1+2\nu)}<0$. Since $\mathcal{E}_k\leq\mathcal{E}_{k-1}$ for all $k> k_1$, then $h(\mathcal{E}_{k-1})\leq h(\mathcal{E}_k)$ for all $k> k_1$. 
We consider two scenarios: in the first case, for any $k > k_1$, there exists an $r_0 \in (1, +\infty)$ such that: $h(\mathcal{E}_k) \leq r_0 h(\mathcal{E}_{k-1})$; in the second case, for some $k > k_1$, regardless  of the choice of $r_0 \in (1, +\infty)$, we have $h(\mathcal{E}_k) > r_0 h(\mathcal{E}_{k-1})$.

    \medskip 

\textbf{{Case 1}}: There exists $r_0\in(1,+\infty)$ such that 
$ h(\mathcal{E}_k)\leq r_0h(\mathcal{E}_{k-1}), \; \forall k> k_1.$
Then, from  \eqref{eqqq} we get:
\begin{align*}
 &   \bar{c}\leq (\mathcal{E}_{k-1}-\mathcal{E}_k)h(\mathcal{E}_{k})\leq (\mathcal{E}_{k-1}-\mathcal{E}_k)r_0h(\mathcal{E}_{k-1})=r_0h(\mathcal{E}_{k-1})\int_{\mathcal{E}_k}^{\mathcal{E}_{k-1}}{1\,ds}\\
    &{\overset{h(\cdot) \text{ decreases}}{\leq}} r_0\int_{\mathcal{E}_k}^{\mathcal{E}_{k-1}}{h(s)\,ds}= r_0\int_{\mathcal{E}_k}^{\mathcal{E}_{k-1}}{s^{-2\nu}\,ds}=\frac{r_0}{1-2\nu}({\mathcal{E}^{1-2\nu}_{k-1}}-{\mathcal{E}^{1-2\nu}_{k}}).
\end{align*}
Since $\nu>\frac{1}{2}$, then $2\nu-1>0$ and multiplying both sides of the above inequality by $\frac{2\nu-1}{r_0}$, we get that:
\[
0<\frac{\bar{c}(2\nu-1)}{r_0}\leq {\mathcal{E}^{1-2\nu}_{k}}-{\mathcal{E}^{1-2\nu}_{k-1}}.
\]

Let us define $\hat{c}=\frac{\bar{c}(2\nu-1)}{r_0}$ and $\hat{\nu}=1-2\nu<0$. We then get:
\begin{equation}\label{need1}
    0<\hat{c}\leq {\mathcal{E}^{\hat{\nu}}_{k}}-{\mathcal{E}^{\hat{\nu}} _{k-1}} \hspace{0.5cm} \forall k> k_1.
\end{equation}
    \newline
\textbf{{Case 2}}: 
For any $r_0 \in (1, +\infty)$, there exists some $k > k_1$ such that $h(\mathcal{E}_k) > r_0h(\mathcal{E}_{k-1})$, which can be expressed as ${\mathcal{E}_k^{-2\nu}} \geq r_0{\mathcal{E}_{k-1}^{-2\nu}}$. This implies:
$r_0{\mathcal{E}_{k-1}^{2\nu}} \geq {\mathcal{E}_k^{2\nu}}.$
By setting $q=\frac{1}{r_0^{\frac{1}{2\nu}}}\in(0,1)$, we obtain,
\[
q\mathcal{E}_{k-1}\geq\mathcal{E}_k.
\]
 Since $\hat{\nu}=1-2\nu<0$ we have $ q^{\hat{\nu}}{\mathcal{E}^{\hat{\nu}}_{k-1}}\leq{\mathcal{E}^{\hat{\nu}}_k}$ and then, it follows that: 
\[(q^{\hat{\nu}}-1){\mathcal{E}^{\hat{\nu}}_{k-1}}\leq{\mathcal{E}^{\hat{\nu}}_{k}}-{\mathcal{E}^{\hat{\nu}}_{k-1}}.
\] 
Since $q\in(0,1)$ and $\hat{\nu}<0$, we have  $q^{\hat{\nu}}-1>0$ and since, moreover, $\mathcal{E}_k\to0^+$ as $k\to\infty$ which is equivalent to $\mathcal{E}^{\hat{\nu}}_k\to+\infty$ as $k\to\infty$ , then there exists $\Tilde{c}$ such that $(q^{\hat{\nu}}-1){\mathcal{E}^{\hat{\nu}}_{k-1}}\geq\tilde{c}$ for all $k> k_1$ (it is sufficient to take $\tilde{c}=(q^{\hat{\nu}}-1){\mathcal{E}^{\hat{\nu}}_{k_1}}$). Therefore, it follows that:
\begin{equation}\label{need2}
        0<\tilde{c}\leq(q^{\hat{\nu}}-1){\mathcal{E}^{\hat{\nu}}_{k-1}}\leq {\mathcal{E}^{\hat{\nu}}_{k}}-{\mathcal{E}^{\hat{\nu}}_{k-1}}.
\end{equation}
 Hence, regardless of the  cases involved (case 1 or case 2), one can always combine \eqref{need1} and \eqref{need2} by choosing ${\mu}=\min\{\hat{c},\tilde{c}\}>0$,  to obtain
\[
        0<{\mu}\leq {\mathcal{E}^{\hat{\nu}}_{k}}-{\mathcal{E}^{\hat{\nu}}_{k-1}} \hspace{0.5cm} \forall k> k_1.
\]
Summing the above inequality from $k_1+1$ to some $k> k_1+1$ gives
\[
\mu\sum_{i=k_1+1}^{k}{1}\leq\sum_{i=k_1+1}^{k}{\left({\mathcal{E}^{\hat{\nu}}_{i}}-{\mathcal{E}^{\hat{\nu}}_{i-1}}\right)}.
\]
By simplifying the summation, we get:
\[
\mu(k-k_1) + {\mathcal{E}^{\hat{\nu}}_{k_1}}\leq  {\mathcal{E}^{\hat{\nu}}_{k}}.
\]

Moreover, since $\hat{\nu}=1-2\nu<0$, it follows that 
\[\mathcal{E}_k\leq (\mu(k-k_1) + {\mathcal{E}^{\hat{\nu}}_{k_1}})^{\frac{1}{\hat{\nu}}}=(\mu(k-k_1) + {\mathcal{E}^{1-2{\nu}}_{k_1}})^{\frac{1}{1-2{\nu}}}.\]
Furthermore,  since $\{\mathcal{E}_k\}_{k\geq k_1}$ is monotonically decreasing, then $\mathcal{E}_{k}\leq\mathcal{E}_{k-1}$, and since $\nu\in(\frac{1}{2},1)$, then $0<1-\nu<\frac{1}{2}$. Hence
\[
 \mathcal{E}_{k}^{1-\nu}\leq \mathcal{E}_{k-1}^{1-\nu}\leq \quad \text{and}\quad  \mathcal{E}_{k-1}^{\frac{1}{2}}\leq \mathcal{E}_{k-1}^{1-\nu}.
\]
Then, it follows that:
\[
 \max\{\mathcal{E}_k^{1-\nu}, \sqrt{\mathcal{E}_{k-1}} \}\leq \max\{\mathcal{E}_{k-1}^{1-\nu}, \sqrt{\mathcal{E}_{k-1}} \}= {\mathcal{E}^{1-\nu}_{k -1}}.
\]
Thus, \eqref{rate_point1} becomes: 
   \[
         \|x_{k}-x_{\rho}^*\|\leq C\left(\frac{1}{\mu(k-k_1-1)+\mathcal{E}_{k_1}^{1-2\nu}}\right)^{\frac{1-\nu}{2\nu-1}}, \hspace{0.5cm} \forall k>k_1.
         \]
         
   \end{enumerate} 
This concludes our proof.\qed
\end{proof}
 {Note that in Theorems \ref{main_result2} and \ref{corollary}, $x_{\rho}^*$ represents a critical point of the Lyapunov function $\mathcal{P}_{\rho}(\cdot)$, which is obtained for any choice of $\rho\geq 1$. However, as indicated in Theorem \ref{main_result}, it is essential to choose the penalty parameter $\rho$  large enough for $x_{\rho}^*$ to qualify as an  $\epsilon$ first order solution to the problem $\eqref{eq1}$. 
Consequently, by following the prescribed choice of $\rho$ in Theorem \ref{main_result} and considering the KL condition, LQP algorithm guarantees convergence to an $\epsilon$ first order solution at the (sub)linear rates determined by Theorem \ref{corollary}.}
 

\begin{remark}
Note that in our convergence  analysis,  in order to find some approximate feasibility bound (see Theorem 1), we used  the following  inequality: 
\begin{equation}\label{licq_SIGMA}
\sigma \|F(x)\| \leq \|J_F^T(x)F(x)\| \quad \forall x \in \mathcal{S},
\end{equation}
for some $ \sigma > 0 $, which follows from the LICQ assumption (see Remark 2). Such condition can be generalized so that our convergence analysis can also cover problems with  simple constraints (i.e., $x \in \mathcal{X} \subseteq \mathbb{R}^n$, where $ \mathcal{X} $ is a nonempty closed convex set that admits an easy projection): 
\begin{align*}
\min_{x \in \mathcal{X}} f(x) \;\;\;  \textrm{s.t.} \;\;  F(x) = 0.
\end{align*}
In this case, the previous regularity condition \eqref{licq_SIGMA} can be replaced in our previous proofs with the existence of some $ \sigma > 0 $ such that (see also \cite{Lu:22}):
\[
\sigma \|F(x)\| \leq \text{dist}\left(-J_F^T(x)F(x), N_{\mathcal{X}}(x)\right) \quad \forall x \in \mathcal{S},
\]
where $ N_{\mathcal{X}}(x) $ denotes the normal cone of  $ \mathcal{X} $ and $\mathcal{S}\subseteq\mathcal{X}$. Note that when $ \mathcal{X} = \mathbb{R}^n $, we have $ N_{\mathcal{X}}(x) = \{0\} $ and the above condition simplifies to  \eqref{licq_SIGMA}. 
\end{remark}

 \subsection{Selection of the penalty parameter $\rho$}
 The above results, showcasing the total number of iterations required to find an $\epsilon$ first-order solution to the problem, rely on the assumption that the penalty parameter $\rho$ exceeds a certain threshold. However, determining this threshold beforehand poses challenges as it depends on unknown parameters of the functions within the problem and the algorithm's parameters.
To overcome this challenge, we propose a scheme that allows us to determine a sufficiently large value of $\rho$ without needing specific parameter information. Inspired by Algorithm 3 in \cite{XieWri:21}, our approach involves repeatedly calling Algorithm \ref{alg1} as an inner loop. If Algorithm \ref{alg1} fails to converge within a predetermined number of iterations, we gradually increase the penalty parameter $\rho$ by a constant multiple in the outer loop. The detailed implementation of this scheme is outlined in Algorithm \ref{alg2}.
Remarkably, in certain applications, such as Model Predictive Control, we need to solve a series of problems with a desired level of accuracy $\epsilon$. Interestingly, the optimal penalty parameter for the initial optimization problem often works well for subsequent problems as well. This implies that investing more computational resources into the first optimization problem can lead to considerable savings for the remaining optimization problems. Thus, this adaptive approach ensures the effectiveness and efficiency of our algorithm, even in scenarios where precise parameter information is unavailable. This adaptability  makes our method a practical and robust solution for a wide range of optimization problems, with potential benefits in various real-world applications.\\
\begin{algorithm}
\caption{LQP method with trial value of $\rho$}\label{alg2}
\begin{algorithmic}[1]
\State  $\textbf{Initialization: } \tau>1, \epsilon>0 \text{ and } \rho_0>0;$
\State $l \gets 0$
\While{$\text{ Infeasible }$}
     \State $\rho_{l+1}\gets \tau \rho_l$
    \State Call Algorithm \ref{alg1} using warm start until  $\|\nabla \mathcal{P}_{\rho_{l+1}}(x^l_{k^*+1})\|\leq\epsilon$ 
    \State $l \gets l+1$
\EndWhile
\end{algorithmic}
\end{algorithm}

\noindent  Algorithm \ref{alg2} is well-defined and terminates in a finite number of iterations, provided that the  parameter $\tau >1$. In fact, during the $l$-th iteration of Algorithm \ref{alg2}, we have $\rho_{l+1} = \tau^{l+1} \rho_0$    and $\lambda^l_{k^*+1} = \rho F(x^l_{k^*+1})$ is bounded by some positive constant $\frac{M}{\sigma}$ (see proof of Theorem \ref{main_result}),
where the superscript \(l\) in \(x^l_{k^*+1}\) and \(\lambda^l_{k^*+1}\) corresponds to the outer iterations of Algorithm \ref{alg2}, while the subscript \(k^*+1\) indicates the inner iterations of Algorithm \ref{alg1} at which the approximate stationarity of the penalty function \(\mathcal{P}_{\rho_{l+1}}(\cdot)\) is achieved.
Consequently,    we have:
\[ 
\|F(x^l_{k^*+1})\| = \frac{\|\lambda^l_{k^*+1}\|}{\tau^{l+1}\rho_0} \leq \frac{M}{\rho_0\sigma\tau^{l+1}}.
\]
 Hence, $\|F(x^l_{k^*+1})\| \leq \epsilon$ provided that  $\frac{M}{\rho_0\sigma\tau^{l+1}} \leq \epsilon$, which is equivalent to $\tau^{l+1} \geq \frac{M}{\rho_0\sigma\epsilon}$. Thus, we require:
\[ 
(l+1)\log\tau \geq \log\frac{M}{\rho_0\sigma} + \log\frac{1}{\epsilon}. 
\]
This leads to the following condition, considering that  $\tau = \frac{1}{\epsilon}$:
\[ l \geq \frac{\log\frac{M}{\rho_0\sigma}}{\log\frac{1}{\epsilon}}. \]
Therefore, for  $\tau = \frac{1}{\epsilon}$, one has to increase $\rho$ at most $\frac{\log\frac{M}{\rho_0\sigma}}{\log\frac{1}{\epsilon}}$ times to reach a value that guarantees the desired precision.


\section{Numerical results}\label{sec5}
In this section we numerically compare Algorithm \ref{alg1} (LQP) with SCP algorithm \cite{MesBau:21}, IPOPT \cite{WacBie:06}, Inexact Restoration algorithm \cite{BueMar:20} and {Algencan \cite{BirMar:14} (an augmented Lagrangian based method)}, on nonconvex optimization problems with nonlinear equality constraints. The simulations are implemented in Python and executed on a PC with (CPU 2.90GHz, 16GB RAM). Since one cannot guarantee that the SCP iterates converge to a first-order (KKT) point, we choose the following stopping criteria: we stop the algorithms when the difference between two consecutive values of the objective function is less than a tolerance $\epsilon_1=10^{-3}$ and the norm of constraints is less than a tolerance $\epsilon_2=10^{-5}$.  For the implementation of our method, we choose $\beta_k$ constant and equal to $10$,  and we select the penalty parameter to be constant for a given problem, specifically $\rho$ is chosen from the interval $[10^7,  10^9]$  (for each test problem, the corresponding value of $\rho$ used in LQP is given in Table 1). A problem is considered to be solved by a method if the stopping criteria are achieved in less than 30 minutes. In the cases when such conditions are not met for an algorithm, we wrote "-". The numerical results are illustrated in Table \ref{tab1} and Figure \ref{fig1}. 

\begin{table}
\small
{
    
      \centering
\begin{adjustbox}{width=\columnwidth,center}

    \begin{tabular}{|c|cc|cc|cc|cc|cc|}
    \hline
   \multirow{3}{*}{\backslashbox{(n,m)}{Alg}} 
     & \multicolumn{2}{c|}{LQP} &
     \multicolumn{2}{c|}{SCP} &
      \multicolumn{2}{c|}{IPOPT}&
      \multicolumn{2}{c|}{Inexact Restoration\cite{BueMar:20}}&
       \multicolumn{2}{c|}{Algencan}\\ \cline{2-11}
              & \# iter     & cpu  
          & \# iter    & cpu &
           \# iter     & cpu &
          \# iter     & cpu  &
          \# iter     & cpu     \\ 
         & $f^*$    & $\|F\|$  
          & $f^*$    & $\|F\|$ &
           $f^*$    & $\|F\|$  &
           $f^*$    & $\|F\|$ &
          $f^*$    & $\|F\|$\\
    \hline

    OPTCTRL3 &  5 &  {0.11}
    &5 &  0.16 & 
      7 & 7.40 & 
     - & -& 
    6 & \textbf{0.01}\\
    (119,80)/$\rho=2e7$ & 2047.99 &  1.43e-06
      &2048.01 & 4.52e-10 & 
      2048.01& 1.84e-08 & 
      -& -& 
    2048 & 3.15 e-10\\\hline
      
   OPTCTRL3 &  7 &  \textbf{0.35} 
    &7 &  1.55 & 
      10 & 11.99& 
     - & -& 
    13 & 3.47\\
    \shortstack{(1199,800)/$\rho=5e7$}& 18459.99 & 4.89e-06
      &18460.22 & 1.84e-09 & 
      18460.22& 6.33e-09 & 
      - & -& 
   18460 & 7.49 e-09\\\hline

    OPTCTRL3&  24 &  \textbf{10.63}
    &24 &  105.08 &  
      11 & 26.95 & 
      -& -& 
    11 & 102.68\\
   (4499,3000)/$\rho=5e7$ & 74464.82 & 1.43e-06
      &74465.03 & 6.87e-09 &
      74465.03& 1.09e-08 & 
      -& -& 
    74470 & 8.66 e-09\\\hline
      
   DTOC4&  3 &  {0.11} 
    &4 &  0.14 & 
      3 & 7.61 & 
      4& 0.22& 
    12 & \textbf{0.08}\\
    (299,198)/$\rho=1e7$ & 2.94 & 2.33e-07
      &2.94 & 3.48e-09 & 
      2.94& 1.02e-08 & 
      2.94& 1.29e-06& 
    2.94 & 8.03 e-09\\\hline
      
    DTOC4 &  3 &  \textbf{0.29} 
    &3 &  0.87 & 
      3 & 12.45 & 
      6& 1.39& 
    19 & 1.53\\
    (1497,998)/$\rho=2e7$ & 2.87 & 1.16e-06
      &2.88 & 1.16e-06 &  
      2.88 & 1.20e-08 & 
      2.88& 6.28e-06& 
    2.88 & 5.12 e-08\\\hline

        DTOC4 &  4 &  \textbf{1.07} 
    &4 &  5.81 &  
      3 & 23.51 & 
      3& 2.33& 
    13 & 4.88\\
    (2997,1998)/$\rho=2e7$& 2.87 & 2.832-07
      &2.87 & 2.83e-07 &  
      2.87& 9.33e-09 & 
      2.87& 2.64e-07& 
    2.87 & 6.59 e-09\\\hline
      
   DTOC4 &  4 &  \textbf{2.87} 
    &3 &  16.74 &  
      3 & 29.02 & 
      4& 6.68& 
    13 & 12.35\\
    (4497,2998)/$\rho=1e8$& 2.87 & 3.02e-07
      &2.87 & 4.87e-10 &  
      2.87& 3.66e-08 & 
      2.87& 3.00e-06& 
    2.87 & 3.56e-08\\\hline

    DTOC4 &  4 &  \textbf{32.30} 
    &4 &  566.80 &  
      3 & 146.73 & 
      -& -& 
    18 & 149.66\\
    (14997,9998)/$\rho=2e8$& 2.87 & 1.05e-07
      &2.87 & 1.05e-07 &  
      2.86& 4.49e-09 & 
      -& -& 
    2.86 & 7.27e-09\\\hline

    DTOC5 &  5 &  {0.12}
    &5 &  0.21 &  
      4 & 10.63 & 
      13& 0.36& 
    9 & \textbf{0.03}\\
       \shortstack{(198,99)/$\rho=1e7$} & 1.53 & 6.88e-08
      &1.53 & 6.64e-08 &  
      1.53& 4.99e-11 & 
      1.53& 7.38e-06& 
    1.53 & 4.93 e-09\\\hline
      
    DTOC5 &  7 &  \textbf{0.29}
    &7 &  0.98 &  
      3 & 12.06 & 
      5& 0.73& 
    19 & 0.52\\
       \shortstack{(998,499)/$\rho=1e7$} & 1.53 & 3.42e-06
      &1.53 & 3.45e-06 &  
      1.53& 7.76e-07 & 
      1.53& 1.47e-06& 
    1.53 & 9.72 e-08\\\hline
      
   DTOC5 &  10 &  \textbf{1.33}
        &10 &  4.52 &  
      3 & 18.74 & 
      6& 1.95& 
    12 & 1.72\\
      \shortstack{ (1998,999)/$\rho=5e7$} & 1.53 & 1.56e-06
      &1.53 & 1.56e-06 &  
      1.53& 6.88e-08 & 
      1.53& 4.02e-06& 
    1.53 & 3.11 e-08\\\hline
      
    DTOC5&  24 &  \textbf{47.03 }
    &24 &  799.07 &  
      3 & 75.25 & 
      16& 137.53& 
    18 & 48.27\\
    (9998,4999)/$\rho=3e8$& 1.54 & 1.97e-07
      &1.54 & 1.96e-07 &  
      1.53& 2.49e-07 & 
      1.53& 3.84e-07& 
    1.53 & 3.31 e-07\\\hline

    DTOC6 &  30 &  {0.30}
    &30 &  0.72 &  
      8 & 8.22 & 
      81& 2.10& 
    14 & \textbf{0.08}\\
       \shortstack{(200,100)/$\rho=1e^8$} & 729.00 & 2.20e-06
      &729.00 & 1.01e-09 &  
      727.98& 3.52e-10 & 
      728.29& 1.14e-06& 
    728 & 2.75 e-09\\\hline
      
   DTOC6 &  32 &  \textbf{1.17}
        &32 &  4.56 &  
      9 & 13.10 & 
      106& 14.96& 
    20 & 1.20\\
      \shortstack{ (1000,500)/$\rho=3e8$} & 6848.01 & 2.97-06
      &6848.06 & 2.7e-10 &  
      6846.61& 7.56e-10 & 
      6847.45& 1.92e-06& 
    6847 & 7.84e-09\\\hline
      
    DTOC6&  33 &  \textbf{3.88 }
    &33 &  15.25 &  
      9 & 18.85 & 
      118& 33.22& 
    17 & 5.08\\
    (2000,1000)/$\rho=3e8$& 17177.66 & 8.95e-06
      &17178.07 & 1.46e-09 &  
      17176.45& 5.24e-09 & 
      17177.56& 1.25e-06& 
    17180 & 2.02 e-09\\\hline

    ORTHREGA &  54 & {0.52}
    & 18 &   {0.47} &  
      66 & 7.37 & 
      -& -& 
    17 & \textbf{0.06}\\
    (133,64)/$\rho=2e7$ &  {350.30} & 2.42e-06
      &414.53 & 1.87e-06 &  
      350.30 & 3.02e-10 & 
      -& -& 
    350.30 & 3.79 e-09\\\hline
      
    ORTHREGA&  55 &  {1.22}
    & 39 &  1.73 &  
      76 & 10.14 & 
      -& -& 
    19 & \textbf{0.28}\\
    (517,256)/$\rho=2e7$& 1414.05 & 1.08e-06
      &1664.80 & 1.24e-06 &  
      1414.05& 6.19e-10 & 
      -& -& 
    1414 & 3.34 e-09\\\hline
      
    ORTHREGA &  94 &  {14.24}
    &67 &  31.78 &  
      14 & 23.99 & 
      -& -& 
    32 & \textbf{10.01}\\
    (2053,1024)/$\rho=5e7$ & 5661.43 & 3.19e-06
      &6654.78 & 2.07e-06 &  
      5661.43& 9.25e-07 & 
      -& -& 
    5661 & 2.17 e-08\\
      

       \hline
   %
       MSS1 & 17  &  \textbf{0.22}
    & 12& 0.15  &  
      53 & 13.52 & 
      42& 0.58& 
    15 & 0.53\\
       \shortstack{(90, 73)/$\rho=1e8$} & -16.00 & 6.04e-06
      &-8.71e-08 & 1.76e-06 &  
      -16.00& 4.17e-08 & 
      -15.48& 3.93e-06& 
    -15.00 & 3.29 e-08\\\hline
      
     
       MSS2 & 19  &  \textbf{2.64}
    & 21& 8.05  &  
      7 & 14.65 & 
      15& 7.75& 
    - & -\\
       \shortstack{(756, 703)/$\rho=1e8$} & -121.03 & 6.20e-06
      & -2.53e-10& 6.12e-06 &  
      -26.97& 5.96e-08 & 
      -0.37& 3.55e-06& 
    - & -\\\hline
      
   
       MSS3 & 68  &  \textbf{78.73}
    & 22& 135.15  &  
      -&- & 
      23& 155.51& 
    - & -\\
       \shortstack{(2070, 1981)/$\rho=1e9$} & -339.69 & 8.17e-06
      &-5.29e-09  & 7.76e-06  &  
      -& - & 
      -1e-03& 7.98e-06& 
    - & -\\\hline
      
     %
       OPTCTRL6 & -  &  -
    & 24&   \textbf{13.46}  &  
      13 & 27.42 & 
      -& -& 
    11 & 101.34\\
       \shortstack{(4499, 3000)} & - & -
      &74465.03 & 3.27e-09 &  
      74465.03& 2.32e-09 & 
      -& -& 
    74470 & 8.47 e-09\\\hline

       OPTCDEG2 & 4  &  \textbf{1.31}
    & 3& 2.71  &  
      4 & 7.92 & 
      3& 4.87& 
    19 & 107.27\\
       \shortstack{(4499, 3000)/$\rho=3e7$} & 59.15 & 3.07e-07
      &59.08 & 5.85e-08 &  
      227.72& 6.55e-08 & 
      59.08& 9.39e-10& 
    227.7 & 2.58 e-07\\\hline

    
       ORTHREGC & 101  &  85.75
    & 29& \textbf{20.93}  &  
      16 & 31.14 & 
      6& 19.62& 
    26 & 17.82\\
       \shortstack{(5005, 2500)/$\rho=5e7$} & 1681.78 & 8.65e-06
      &94.81 & 8.42e-06 &  
      94.81& 7.52e-07 & 
      133.86& 8.68e-08& 
    94.81 & 3.07e-07\\\hline

     
       EIGENB2 & 11  &  18.76
    &5 &  5.28 &  
      27 & \textbf{56.61} & 
      9& 53.24& 
    6 & 303.18\\
       \shortstack{(2550, 1275)/$\rho=7e8$} & 110.50 & 3.72e-13
      &110.50 & 1.61e-14 &  
      0.00& 5.45e-09 & 
      110.49& 1.12e-08& 
    0.00 & 8.33e-08\\\hline

       EIGENC2 & 6  &  3.12
    & 6 & 4.68  &  
      13 & \textbf{24.93} & 
      10& 83.99& 
    6 & 32.02\\
       \shortstack{(2652, 1326)/$\rho=7e8$} & 11162.34  & 3.25e-14
      & 11162.75& 4.64e-16 &  
      0.00& 8.43e-10 & 
     11150.50& 2.12e-10& 
    0.00 & 3.51e-10\\\hline
      
     
       EIGENACO & 19  &  {9.98}
    & 8& 1.75  &  
      - & - & 
      14& 10.52& 
    2 & \textbf{1.87}\\
       \shortstack{(2550, 1275)/$\rho=7e8$} & 10731.01 & 1.00e-15
      &22425.04 & 2.37e-18 &  
      -& - & 
      10731.25& 1.54e-09& 
    0.00 & 3.21e-09\\\hline

       EIGENBCO & 6  &  11.80
    & 5 & 1.23  &  
      9 & \textbf{19.58} & 
      11&50.84 & 
    - & -\\
       \shortstack{(2550, 1275)/$\rho=7e8$} & 49.50 & 4.19e-14
      &49.50 & 5.79e-16 &  
      0.00& 3.18e-17 & 
      59.50& 8.11e-07& 
    - & -\\\hline
      
     
       EIGENCCO & 7  &  4.06
    & 7& 1.55  &  
      13 & \textbf{42.88} & 
      7& 125.49& 
    6 & 1203.96\\
       \shortstack{(2652, 1326)/$\rho=3e8$} & 11100.50 & 7.56e-11
      &11100.51 & 1.99e-10 &  
      0.00& 2.75e-12 & 
      11100.50& 2.01e-10& 
    0.00 & 1.86e-10\\\hline

       DTOC1NA & 50  &  50.67
    & 4& 3.86  &  
      5 & {12.08} & 
      14& 40.81& 
    5 & \textbf{0.23}\\
       \shortstack{(5994, 3996)/$\rho=7e8$} & 47.64 & 5.95e-11
      &47.66 & 5.03e-13 &  
      4.15& 7.44e-11 & 
      47.64& 2.96e-09& 
    4.14 & 8.03e-10\\\hline

     
       DTOC1NB & 46  &  46.65
    & 4& 3.89  &  
      5 &{11.71} & 
      13& 38.32& 
    5 &  \textbf{0.36}\\
       \shortstack{(5994, 3996)/$\rho=1e9$} & 48.45 & 6.94e-11
      &48.47 & 1.68e-14 &  
      7.13& 6.31e-12 & 
      48.44& 2.46e-08& 
    7.14 & 1.19e-10\\
      

       \hline

       DTOC1NC & 27  &  28.36
    & 6 & 5.76  &  
      3 & {8.12} & 
      6& 17.78& 
    7 & \textbf{0.48}\\
       \shortstack{(5994, 3996)/$\rho=7e8$} & 58.63 & 1.43e-09
      & 58.64 & 3.20e-11 &  
      35.21& 5.80e-10 & 
      58.64& 1.67e-07& 
    35.20 & 6.42e-09\\\hline

     
       DTOC1ND & 22  &  24.04
    & 7& 6.70  &  
      3 & {12.79} & 
      6 & 17.75& 
    4 & \textbf{0.26}\\
       \shortstack{(5994, 3996)/$\rho=5e8$} & 66.65 & 1.81e-11
      &66.66 & 3.38e-11 &  
      47.63& 7.85e-09 & 
      66.67& 6.36e-07& 
    47.60 & 1.76e-10\\
      

       \hline

       SPINOP & 11  &  \textbf{4.25}
    & -& -  &  
      - & - & 
      7& 6.23& 
    - & -\\
       \shortstack{(1327, 1325)/$\rho=3e7$} & 2998.78 & 4.10e-07
      &- & - &  
      -& - & 
      3000.33& 1.06e-10& 
    - & -\\\hline

       DIXCHLNV & 114  &  54.32
    & -& -   &  
      28 & \textbf{47.31} & 
      8& 3.29& 
    - & -\\
       \shortstack{(1000, 500)/$\rho=3e7$} & 4.87 & 8.94e-07
      &- & - &  
      0.00& 3.58e-08 & 
      12.53& 8.28e-08& 
    - & -\\\hline

       GLIDER & 69  &  136.18
    & -& -  &  
      53 & \textbf{107.53} & 
      -& -& 
    - & -\\
       \shortstack{(5207, 4808)/$\rho=2e8$} & -543.46 & 8.27e-07
      &- & - &  
      -1247.97& 9.01e-08 & 
      -& -& 
    - & -\\\hline

       DTOC2 & 11  &  \textbf{12.16}
    & 21& 19.73  &  
      11 & 22.64 & 
      15& 44.98& 
    21 & 105.04\\
       \shortstack{(5994, 3996)/$\rho=2e8$} & 0.52 & 3.39e-06
      &0.91 & 9.66e-06 &  
      0.50& 6.21e-09 & 
      0.52& 9.51e-06& 
    0.51 & 6.28e-09\\\hline

       ROBOTARM & 12  &  9.18
    & -& -  &  
      7 & \textbf{10.92} & 
      -& -& 
     23 & 377.26\\
       \shortstack{(4400, 3202)/$\rho=1e7$} & 18.34 & 2.88e-07
      &- & - &  
      9.14& 2.05e-08 & 
      -& -& 
    9.14 & 1.20e-08\\\hline

       ROCKET & 251  &  142.92
    & -& -  &  
      5 & 7.49 & 
      11& \textbf{6.90}& 
    - & -\\
       \shortstack{(2403, 2002)/$\rho=5e7$} & -1.00 & 2.24e-06
      &- & - &  
      -1.00& 2.73e-07 & 
      -0.99& 3.23e-06& 
    - & -\\\hline

       CATMIX & 12  &  {2.65}
    & 9 & 2.84  &  
      3 & 4.36 & 
      4& \textbf{2.30}& 
    22 & 26.94\\
       \shortstack{(2401, 1600)/$\rho=1e8$} & -0.03 & 7.32e-08
      &-0.03 & 5.71e-09 &  
      -0.04& 4.88e-09 & 
      -0.34&3.19e-07 & 
    -0.04 & 2.55e-09 \\\hline
      
    \end{tabular}%
    \end{adjustbox}
}
\caption{Comparison between LQP, SCP, IPOPT, Inexact Restoration and Algencan on selected problems from the CUTEst collection.}
\label{tab1}
\end{table}

\medskip 

\noindent In Table \ref{tab1}, we compare the number of iterations, CPU time (sec), objective value, and feasibility violation (the Euclidean norm of functional constraints) for LQP, SCP, IPOPT,  Inexact Restoration and Algencan on real problems selected from the CUTEst collection \cite{GouOrb:15}. We have chosen all problems from CUTEst with nonlinear equality constraints, excluding obvious cases (small number of constraints and problems with only linear equality constraints). Note that for many test cases from CUTEst, LQP algorithm is capable of yielding solutions much faster than the other methods {(the best cpu time obtained by a method is written in bold in the table)}. However, for few problems, the optimal values yielded  by LQP,  SCP and the Inexact Restoration methods are not the best, as those given by IPOPT and Algencan. We attribute this phenomenon to the problem structure, as observed in the test group "EIGEN" and certain instances in "DTOC" test group. Another possible reason, specifically for LQP, could be the optimal objective value being much smaller than the constraint penalty, as we suspect is the case for the problem instance "ORTHREGC". Moreover, from Table \ref{tab1}, it can be seen that our method fails only in one case, 'OPTCTRL6', where it cannot provide some information about the solution. In contrast, the other methods fail in more than one case, thus proving that LQP is robust compared to SCP, IPOPT, the Inexact Restoration methods and Algencan.

\begin{figure}[htp]
   \hspace{-1.1cm} \includegraphics[width=14cm]{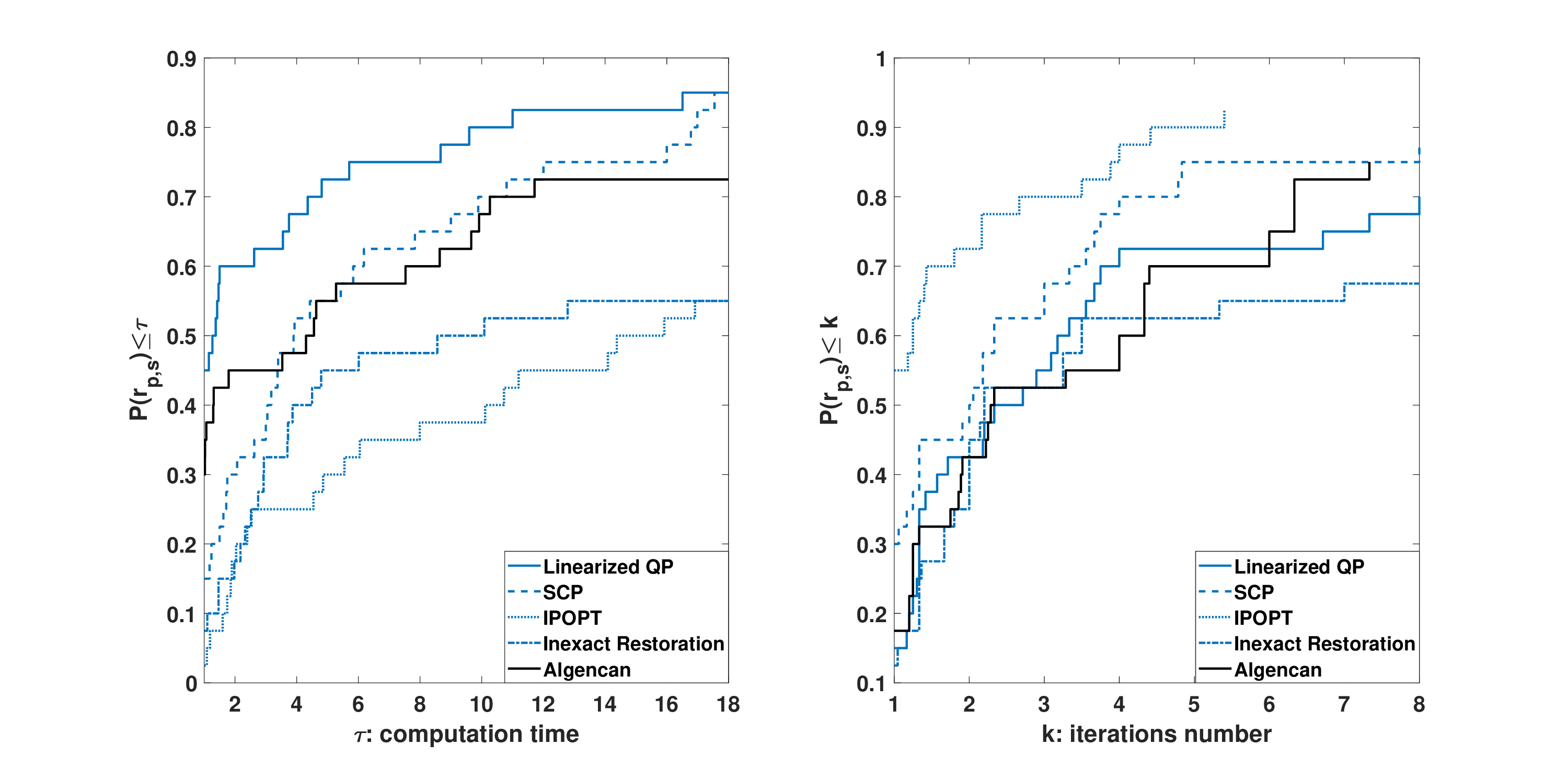}
    \caption{Performance profiles for  computation time (left) and number of iterations (right).}
    \label{fig1}
\end{figure}

\medskip 

\noindent Figure \ref{fig1} shows the performance profiling for computation time (left) and number of iterations (right) for the 5 algorithms. In the performance profiling, the vertical axis $P(r_{p,s}\leq\tau)$ ($P(r_{p,s}\leq k)$) represent the proportion of problems in the numerical experiments for which $r_{p,s}$ does not exceed $\tau$ ($k$), respectively, where $r_{p, s}$ is the ratio of the computational time (number of iterations) that the $s$ solver takes to solve problem $p$ to the shortest computational time (number of iterations) among the five algorithms to solve problem $p$, respectively. It is clear from the computational time profile,   Figure \ref{fig1} (left), that the proposed algorithm, LQP, approaches the limit line $1$ faster than SCP,  IPOPT, Inexact Restoration and Algencan algorithms. However, for the number of iterations this is not always the case. \\

\medskip 

\noindent Based on these preliminary experiments, one can see that LQP algorithm is an efficient and viable  method. Its ease of implementation, theoretical guarantees for convergence and practical performance  make it suitable for addressing large and hard  optimization problems with nonlinear equality constraints. }


\section{Conclusions}\label{sec6}
In this paper, we have proposed a linearized quadratic penalty  method for solving (locally) smooth  optimization problems with nonlinear equality constraints. In this method  we have linearized the cost function and functional  constraints  within the quadratic penalty  function and added a regularization term. By dynamically generating the regularization (proximal) parameter,  we have derived  global convergence rate to an $\epsilon$ first-order optimal solution and improved local convergence rates under the KL condition. Moreover, we have numerically shown that the proposed algorithm is efficient, comparing it with four known algorithms,  SCP,  IPOPT, Inexact Restoration and Algencan. 

\medskip 

\noindent  Furthermore, a potential avenue for extending this work involves the incorporation of full second-order derivatives of both the objective function and constraints. We believe that the analysis of such a scheme is viable, following a similar rationale as presented in this paper, but the subproblem may become nonconvex in this case. 
Finally, it would be interesting to extend this method to the augmented Lagrangian framework. This extension would preserve the simplicity of the subproblems and  would allow for controlling feasibility through dual variables rather than the penalty parameter, potentially leading to smaller values for this parameter.

\vspace{-0.4cm}


\section*{Conflict of interest}

\vspace{-0.3cm}

The authors declare that they have no conflict of interest.

\vspace{-0.5cm}

\section*{Data availability}

\vspace{-0.3cm}

It is not applicable.


\vspace{-0.4cm}

\appendix
\section*{Appendix}

\vspace{-0.3cm}

\noindent {\textbf{Proof of Lemma \ref{lemma3}}.
 Note that the subproblem's objective function $x \mapsto \bar{\mathcal{P}}_{\rho}(x\cdot;x_k)+\frac{\beta_{k+1}}{2}\|x-x_k\|^2$ is strongly convex with strong convexity constant $\beta_{k+1}$. Combining this with  the optimality of $x_{k+1}$  and the fact that $\bar{\mathcal{P}}_{\rho}(x_{k};x_k)=\mathcal{P}_{\rho}(x_{k})$, we get: 
\begin{equation} \label{initial_decrease}
 \bar{\mathcal{P}}_{\rho}(x_{k+1};x_k)\leq \mathcal{P}_{\rho}(x_{k}) -\beta_{k+1}\|x_{k+1}-x_k\|^2.
\end{equation}

\vspace{-0.1cm}

\noindent Further, since $f$ has Lipschitz gradient, we have:

\vspace{-0.3cm}

\begin{equation}\label{smooth_object}
f(x_{k+1}) - l_f(x_{k+1};x_k) \leq \frac{L_f}{2}\|x_{k+1}-x_k\|^2. 
\end{equation}

\vspace{-0.1cm}

\noindent Moreover, using properties of the norm, we obtain: 
\begin{align}
   &\frac{\rho}{2}\|F(x_{k+1})\|^2 - \frac{\rho}{2}\|l_F(x_{k+1};x_k)\|^2 \nonumber\\
   & =   \langle \rho l_F(x_{k+1};x_k), F(x_{k+1})-l_F(x_{k+1};x_k) \rangle + \frac{\rho}{2}\|F(x_{k+1})-l_F(x_{k+1};x_k)\|^2 \nonumber\\
   & \leq    \rho \|l_F(x_{k+1};x_k)\| \|F(x_{k+1})-l_F(x_{k+1};x_k) \| + \frac{\rho}{2}\|F(x_{k+1})-l_F(x_{k+1};x_k)\|^2 \nonumber\\
    &\overset{\text{Ass. } \ref{assump2}}{\leq} \rho \|l_F(x_{k+1};x_k)\| \frac{L_F}{2}\|\Delta x_{k+1}\|^2 + \frac{\rho}{2} \left(\frac{L_F}{2}\|\Delta x_{k+1}\|^2\right)^2.  \label{decrease_feasib}
\end{align}

\vspace{-0.2cm}

\noindent  Using the relation between arithmetic and geometric means we  bound $ \rho \|l_F(x_{k+1};x_k)\|$ as:

\vspace{-0.3cm}

\begin{align}
&\rho \|l_F(x_{k+1};x_k)\| - \frac{ \beta_{k+1} - L_f}{2L_F} \leq  \frac{L_F\rho^2}{ 2(\beta_{k+1} - L_f)}    \| l_F(x_{k+1};x_k)\|^2  \nonumber\\
&= \frac{L_F\rho}{ \beta_{k+1} - L_f} \left( \bar{\mathcal{P}_{\rho}} (x_{k+1};x_k) - f(x_{k+1}) + f(x_{k+1}) - l_f(x_{k+1};x_k)  \right)  \nonumber\\
& \overset{\eqref{initial_decrease}, \eqref{smooth_object}}{\leq }\frac{L_F\rho}{\beta_{k+1} - L_f} \left( \mathcal{P}_{\rho}(x_k) - f(x_{k+1}) - \frac{2\beta_{k+1} - L_f}{2}\|\Delta x_{k+1}\|^2 \right)  \nonumber\\
& \leq \frac{L_F\rho}{\beta_{k+1} - L_f} \left( \mathcal{P}_{\rho}(x_k) - f(x_{k+1}) - \left(\beta_{k+1} - L_f\right)\|\Delta x_{k+1}\|^2 \right) \nonumber\\
&\overset{(f(x_{k+1})\geq \bar{f})}{\leq} \frac{L_F\rho}{\beta_{k+1} - L_f} \left( \mathcal{P}_{\rho}(x_k) - \bar{f} \right) - L_F\rho\|\Delta x_{k+1}\|^2  \nonumber \\
& \overset{\eqref{eq_assu}}{\leq} \frac{\beta_{k+1}-L_f}{2L_F} - L_F\rho\|\Delta x_{k+1}\|^2. \label{bound_grad_norm_squared}
\end{align}

\noindent Using \eqref{bound_grad_norm_squared} in \eqref{decrease_feasib}, we get:

\vspace{-0.3cm}

\begin{align}
  \frac{\rho}{2}\|F(x_{k+1})\|^2 - \frac{\rho}{2}\|l_F(x_{k+1};x_k)\|^2  & \leq   \frac{\beta_{k+1}-L_f}{2} \|\Delta x_{k+1}\|^2 - \left(\frac{\rho L_F^2}{2} - \frac{\rho L_F^2}{8} \right) \|\Delta x_{k+1}\|^4 \nonumber\\
  & \leq  \frac{\beta_{k+1}-L_f}{2} \|\Delta x_{k+1}\|^2. \label{to_use_next}
\end{align}
Moreover, we have:

\vspace{-0.3cm}

\begin{align*} 
   \mathcal{P}_{\rho}(x_{k+1}) - \bar{\mathcal{P}}_{\rho}(x_{k+1};x_k) &=f(x_{k+1}) - l_f(x_{k+1};x_k) + \frac{\rho}{2}\|F(x_{k+1})\|^2 - \frac{\rho}{2}\|l_F(x_{k+1};x_k)\|^2.
\end{align*}
Using  \eqref{smooth_object} and \eqref{to_use_next} in the previous relation, it follows that:

\vspace{-0.3cm}

\[
 \mathcal{P}_{\rho}(x_{k+1}) \leq \bar{\mathcal{P}}_{\rho}(x_{k+1};x_k) + \frac{\beta_{k+1}}{2} \|\Delta x_{k+1}\|^2.
\]
Finally, using \eqref{initial_decrease}, we get the decrease in \eqref{decrease}.
This proves our statement.\qed
}

\medskip

\noindent \textbf{Proof of Lemma \ref{bounded_gradient}.} 
Using the optimality condition for $x_{k+1}$, we have:

\vspace{-0.2cm}

\[
    \nabla f(x_{k})+{J_F(x_{k})}^T(\rho l_F(x_{k+1};x_k))+\beta_{k+1}(x_{k+1}-x_{k})=0.
\]

\noindent It then follows, by exploiting the definition of $\mathcal{P}_{\rho}$ and the properties of the derivative, that:

\vspace{-0.3cm}

\begin{align*}
     \nabla\mathcal{P}_{\rho}(x_{k+1})=&\nabla f(x_{k+1})+{J_F(x_{k+1})}^T\big(\rho F(x_{k+1})\big)\\
     = & \nabla f(x_{k+1})-\nabla f(x_k)+\rho\big({J_F(x_{k+1})}-{J_F(x_{k})}\big)^TF(x_{k+1})-\beta_{k+1}\Delta x_{k+1}\\
     & +\rho J_F(x_{k})^T\big( F(x_{k+1})-l_F(x_{k+1;x_k})\big).
\end{align*}

\vspace{-0.1cm}

\noindent Using the triangle inequality, we further get:

\vspace{-0.3cm}

\begin{align} \label{another_bound_grad}
    \|\nabla\mathcal{P}_{\rho}(x_{k+1})\| \leq & \|\nabla f(x_{k+1})-\nabla f(x_k)\|+\rho\|F(x_{k+1})\|\|{J_F(x_{k+1})}-{J_F(x_{k})}\| \nonumber \\
     & +\beta_{k+1}\|\Delta x_{k+1}\| +\rho\|J_F(x_{k})\|\| F(x_{k+1})-l_F(x_{k+1};x_k)\| \nonumber\\
     {\overset{{\text{Ass. } \ref{assump2}}}{\leq}}& \left(L_f+\rho\|F(x_{k+1})\| L_F +\beta_{k+1}\right)\|\Delta x_{k+1}\| + \frac{\rho M_FL_F}{2} \|\Delta x_{k+1}\|^2, 
\end{align}

\vspace{-0.2cm}

\noindent where in the second inequality we use that  $\| F(x_{k+1})-l_F(x_{k+1};x_k)\|\leq L_F/2\|\Delta x_{k+1}\|^2$.
Furthermore, we have:
\begin{align*}
    \frac{\rho}{2}\|F(x_{k+1})\|^2 &= \mathcal{P}_{\rho}(x_{k+1}) - f(x_{k+1})\\
    & \leq \mathcal{P}_{\rho}(x_{k+1}) - \bar{f}.
\end{align*}
It then follws that:
\[
\rho\|F(x_{k+1})\| \leq \sqrt{2\rho} \sqrt{\mathcal{P}_{\rho}(x_{k+1}) - \bar{f}} \overset{\eqref{decrease}}{\leq} \sqrt{2\rho} \sqrt{\mathcal{P}_{\rho}(x_{k}) - \bar{f}}.
\]
Therefore, we get:
\begin{align*}
     \|\nabla\mathcal{P}_{\rho}(x_{k+1})\| & \leq \left(L_f+ L_F \sqrt{2\rho} \sqrt{\mathcal{P}_{\rho}(x_{k}) - \bar{f}} + \beta_{k+1}\right)\|\Delta x_{k+1}\| + \frac{\rho M_FL_F}{2} \|\Delta x_{k+1}\|^2.
\end{align*}
Moreover, using in \eqref{another_bound_grad} the fact that   $\| F(x_{k+1})-l_F(x_{k+1};x_k)\|\leq \|F(x_{k+1})-F(x_k)\|+ \|J_F(x_k)\Delta x_{k+1}\|\leq 2 M_F \|\Delta x_{k+1}\|$, the second claim follows.\qed

\medskip

\noindent \textbf{Proof of Lemma \ref{added_lemma}.}
\eqref{lem_item1} Since $\{x_k\}_{k\geq0}$ is bounded, then there exists   a convergent subsequence  $\{x_k\}_{k\in\mathcal{K}}$ such that $\lim_{k\in\mathcal{K}}{x_k}=x_{\rho}^*$. Hence $\Omega$ is nonempty. Moreover, $\Omega$ is compact since it is bounded and closed.
On the other hand, for any $x_{\rho}^*\in\Omega$, there exists a sequence of increasing integers $\mathcal{K}$ such that $\lim_{k\in\mathcal{K}}{x_k}=x_{\rho}^*$ and using \eqref{bounde_gradient} and $\lim_{k} \|\Delta x_k\|=0$ (see proof of Theorem \ref{main_result}), it follows that:

\vspace{-0.2cm}

\[
\|\nabla\mathcal{P}_{\rho}(x_{\rho}^*)\|=\lim_{k\in\mathcal{K}}{\|\nabla\mathcal{P}_{\rho}(x_k)\|}=0.
\]

\vspace{-0.1cm}

\noindent Hence, $x_{\rho}^*\in\text{crit }\mathcal{P}_{\rho}$ and $0\leq \lim_{k\to\infty}{\text{dist}(x_k,\Omega)}\leq\lim_{k\in\mathcal{K}}{\text{dist}(x_k,\Omega)}=\text{dist}(x_{\rho}^*,\Omega)=0$.\\
\noindent \eqref{lem_item2} Since $\mathcal{P}_{\rho}(\cdot)$ is continuous and $ \{\mathcal{P}_{\rho}(x_k)=P_k\}_{k\geq0}$ converges to $P^{*}$, then any subsequence $\{\mathcal{P}_{\rho}(x_k)=P_k\}_{k\in\mathcal{K}}$ that converges has the same limit $P^{*}$.
This concludes our proof.\qed

\medskip

\noindent \textbf{Proof of Lemma \ref{finite_length}.} From Lemma \ref{lemma3} and \eqref{bar_gamma}, we have: 
\vspace*{-0.2cm}
\begin{align}\label{llyap}
     P_{k+1}-P_{k}&{\overset{{\text{Lemma } \ref{lemma3}}}{\leq}}-\frac{\ubar{\beta}}{2}\|\Delta x_{k+1}\|^2.
\end{align}

\vspace{-0.2cm}

\noindent  Since $ P_{k}\to P^{*}$ and  $\{P_{k}\}_{k\geq0}$ is  decreasing to $P^{*}$, then  the error sequence $\{\mathcal{E}_{k}\}_{k\geq0}$ is nonnegative, monotonically decreasing and converges to $0$. We distinguish  two cases.

\vspace{0.1cm}

\noindent \textbf{{Case 1}}: There exists  $k_1\geq 1$ such that $\mathcal{E}_{k_1}=0$. Then, $\mathcal{E}_{k}=0 \;  \forall k\geq k_1$ and it follows that:

 \vspace{-0.2cm}
 
\[
\|x_{k+1}-x_{k}\|^2\leq\frac{2}{\ubar{\beta}}(\mathcal{E}_{k}-\mathcal{E}_{k+1})=0 \hspace{0.3cm}\forall k\geq k_1.
\]

\vspace{-0.15cm}

\noindent Since the sequence $\{x_{k}\}_{k\geq0}$ is bounded:
 $\sum_{k=1}^{\infty}{\|\Delta x_{k}\|}=\sum_{k=1}^{k_1}{\|\Delta x_{k}\|} <\infty.$

 \vspace{0.1cm}
 
 \noindent \textbf{{Case 2}}: The error $\mathcal{E}_{k}>0 \;  \forall k\geq0$. Then,  there exists  $k_1=k_1(\epsilon,\tau)\geq 1$  such that $\forall k\geq k_1$ we have $\text{dist}(x_k,\Omega)\leq \epsilon$,  $P^{*}<\mathcal{P}_{\rho}(x_k)<P^{*}+\tau$
 and

 \vspace{-0.2cm}
 
 \begin{equation}\label{KL}
     \varphi'(\mathcal{E}_{k})\|\nabla\mathcal{P}_{\rho}(x_{k})\|\geq1,
 \end{equation}
\noindent where $ \epsilon>0, \tau>0$ and $\varphi\in\Psi_{\tau}$ are well defined and correspond to those in Definition \ref{def2}, recall that  $\mathcal{P}_{\rho}(\cdot)$ satisfies the KL property on $\Omega$. Since $\varphi$ is concave, we have $\varphi(\mathcal{E}_{k})-\varphi(\mathcal{E}_{k+1})\geq\varphi'(\mathcal{E}_{k})(\mathcal{E}_{k}-\mathcal{E}_{k+1})$. Then, from \eqref{llyap} and \eqref{KL} we get:
 \begin{align*}
 \|x_{k+1}-x_{k}\|^2
 &\leq\varphi'(\mathcal{E}_{k})\|x_{k+1}-x_{k}\|^2\|\nabla\mathcal{P}_{\rho}(x_{k})\|\nonumber\\
 &\leq\frac{2}{\ubar{\beta}}\varphi'(\mathcal{E}_{k})(\mathcal{E}_{k}-\mathcal{E}_{k+1})\|\nabla\mathcal{P}_{\rho}(x_{k})\|\leq\frac{2}{\ubar{\beta}}\Big(\varphi(\mathcal{E}_{k})-\varphi(\mathcal{E}_{k+1})\Big)\|\nabla\mathcal{P}_{\rho}(x_{k})\|.
 \end{align*}

 \vspace{-0.2cm}
 
\noindent  Note that for given $a,b,c\geq0$, if we have $ {a^2}\leq2 b\times c$, and recognizing that we always have $2 b\times c\leq(b+c)^2$, then  $ {a^2}\leq (b+c)^2$, which in turn implies $a\leq b+c$.  It follows that for any $\theta>0$, we have (by taking $a=\|\Delta x_{k+1}\|$, $b=\frac{\theta}{\ubar{\beta}}\Big(\varphi(\mathcal{E}k)-\varphi(\mathcal{E}{k+1})\Big)$, and $c=\frac{1}{\theta}\|\nabla\mathcal{P}_{\rho}(x_k)\|$):
\begin{align}\label{lmit}
    {\|\Delta x_{k+1}\|}\leq&\frac{\theta}{\ubar{\beta}}\Big(\varphi(\mathcal{E}_{k})-\varphi(\mathcal{E}_{k+1})\Big)+\frac{1}{\theta}\|\nabla\mathcal{P}_{\rho}(x_{k})\|.
\end{align}

 \vspace{-0.2cm}
 
\noindent Furthermore, we have:
$
    \| \nabla{P}(x_{k})\|
    {\overset{{}}{\leq}}\Gamma_{\text{max}}\|\Delta x_{k}\|.
$
Then, \eqref{lmit} becomes:
\begin{align*}
    {\|\Delta x_{k+1}\|}\leq&\frac{\theta}{\ubar{\beta}}\Big(\varphi(\mathcal{E}_{k})-\varphi(\mathcal{E}_{k+1})\Big)+\frac{\Gamma_{\text{max}}}{\theta}\|\Delta x_{k}\|.
\end{align*}

 \vspace{-0.2cm}
 
\noindent Let us now choose $\theta>0$ so that $0<\frac{\Gamma_{\text{max}}}{\theta}<1$ and define the parameter $\delta_0$ as: $\delta_0=1-\frac{\Gamma_{\text{max}}}{\theta}>0$. Then, by
summing up the above inequality from $k=\ubar{k}$ to $k=K$ and using the property: $\sum_{k=\ubar{k}}^{K}{\|\Delta x_{k}\|}=\sum_{k=\ubar{k}}^{K}{\|\Delta x_{k+1}\|}+\|\Delta x_{\ubar{k}}\|-\|\Delta x_{{K+1}}\|$, we get: 
\begin{align*}
   \sum_{k=\ubar{k}}^{K}&\|\Delta x_{k+1}\| \leq \frac{\theta}{\ubar{\beta}\delta_0}\Big(\varphi(\mathcal{E}_{\ubar{k}})-\varphi(\mathcal{E}_{K+1})\Big)+\frac{\Gamma_{\text{max}}}{\theta\delta_0}\|\Delta x_{\ubar{k}}\| -\frac{\Gamma_{\text{max}}}{\theta\delta_0}\|\Delta x_{{K+1}}\|.
\end{align*}

 \vspace{-0.2cm}
 
\noindent Using the fact that  $\{\mathcal{E}_{k}\}_{k\geq k_1}$ is monotonically decreasing and that the function $\varphi$ is positive and increasing, which  yields $\varphi(\mathcal{E}_{k})\geq\varphi(\mathcal{E}_{k+1})>0$,  then:
\begin{align*}
    \sum_{k=\ubar{k}}^{K}{\|\Delta x_{k+1}\|}\leq&\frac{\theta}{\ubar{\beta}\delta_0}\varphi(\mathcal{E}_{\ubar{k}})+\frac{\Gamma_{\text{max}}}{\theta\delta_0}\|\Delta x_{\ubar{k}}\|.
\end{align*}
It is clear that the right-hand side of the above inequality is bounded for any $K\geq\ubar{k}$. Letting  $K\to\infty$, we get that:
$
    \sum_{k=\ubar{k}}^{\infty}{\|\Delta x_{k+1}\|}<\infty.
$
Since the sequence  $\{x_{k}\}_{k\geq0}$ is bounded, it follows that:
$
    \sum_{k=1}^{\ubar{k}}{\|\Delta x_{k}\|}<\infty.
$
Hence: $ \sum_{k=1}^{\infty}{\|\Delta x_{k}\|}<\infty$. 
Let $m, n\in\mathcal{Z}_{+}$ such that $n\geq m$, we have:
\begin{align*}
    \|x_n-x_m\|=\|\sum_{k=m}^{n-1}{\Delta x_{k+1}}\|
    \leq\sum_{k=m}^{n-1}{\|\Delta x_{k+1}\|}.
\end{align*}
Since  $ \sum_{k=0}^{\infty}{\|\Delta x_{k+1}\|}<\infty$, it follows that $\forall \varepsilon>0, \exists N\in\mathcal{Z}_{+}$ such that $\forall m, n\geq N$ where $n\geq m$, we have: $ \|x_n-x_m\|\leq\varepsilon$. This implies that $\{x_k\}_{k\geq0}$ is a Cauchy sequence and converges. Moreover, by Theorem \ref{main_result}, $\{x_k\}_{k\geq0}$  converges to a critical point of $\mathcal{P}_{\rho}(\cdot)$.\qed

\end{document}